\documentclass[14pt, dvipsnames]{extarticle}
\usepackage[left=2cm, top=2cm, right=2cm, bottom=20mm]{geometry} 

\usepackage[T1]{fontenc}
\usepackage[cmintegrals,cmbraces]{newtxmath}
\usepackage[lining]{ebgaramond}
\usepackage{ebgaramond-maths}
\usepackage{euscript}
\usepackage{xcolor}
\usepackage{graphicx}
\usepackage[figurename=Fig.]{caption}

\usepackage[tracking = true, letterspace = 50]{microtype}
\SetTracking{encoding=T1, shape=sc}{50}
\SetTracking{encoding=T1}{50}

\usepackage{amsthm} 
\usepackage{hyperref} 
\usepackage[hyphenbreaks]{breakurl}

\theoremstyle{plain}
\newtheorem{theorem}{Theorem}
\newtheorem{coroll_theorem}{Corollary}[theorem]

\newtheorem{prop}{Proposition}

\newtheorem*{remark*}{Remark}

\theoremstyle{definition}
\newtheorem{defi}{Definition}

\newtheorem{example}{Example}

\theoremstyle{remark}

\usepackage{enumerate}

\usepackage{etoolbox}
\patchcmd{\thmhead}{(#3)}{#3}{}{}
\makeatletter
\renewenvironment{proof}[1][\proofname]{
  \par\pushQED{\qed}\normalfont
  \topsep6\p@\@plus6\p@\relax
  \trivlist\item[\hskip\labelsep\bfseries\itshape #1\@addpunct{.}]
}{\popQED\endtrivlist\@endpefalse}
\makeatother

\let\phi\varphi

\DeclareMathOperator{\kk}{\Bbbk}

\DeclareMathOperator{\Hh}{\mathfrak{H}}

\DeclareMathOperator{\pt}{\mathrm{pt}}

\DeclareMathOperator{\End}{\mathrm{End}}

\DeclareMathOperator{\Sss}{\mathbb{S}}

\DeclareMathOperator{\pr}{\mathrm{pr}}

\DeclareMathOperator{\Q}{\mathbb{Q}}

\DeclareMathOperator{\E}{\EuScript{E}}

\DeclareMathOperator{\G}{\mathbb{G}}

\DeclareMathOperator{\id}{\mathrm{id}}

\DeclareMathOperator{\Sym}{\mathrm{Sym}}

\DeclareMathOperator{\inv}{\mathrm{inv}}

\DeclareMathOperator{\Z}{\mathbb{Z}}
\DeclareMathOperator{\Cc}{\mathbb{C}}

\DeclareMathOperator{\R}{\mathbb{R}}

\DeclareMathOperator{\Oo}{\EuScript{O}}

\DeclareMathOperator{\Spec}{\mathrm{Spec}}

\DeclareMathOperator{\dR}{\mathrm{dR}}

\DeclareMathOperator{\tr}{\mathrm{tr}}

\DeclareMathOperator{\sgn}{\mathrm{sgn}}

\DeclareMathOperator{\GL}{\mathrm{GL}}

\DeclareMathOperator{\SL}{\mathrm{SL}}

\DeclareMathOperator{\Mat}{\mathrm{Mat}}

\DeclareMathOperator{\CP}{\mathbb{C}P}
\DeclareMathOperator{\Hhh}{\mathrm{H}}
\DeclareMathOperator{\Hol}{\mathrm{Hol}}
\DeclareMathOperator{\Mm}{\EuScript{M}}

\DeclareMathOperator{\Sol}{\mathrm{Sol}}
\DeclareMathOperator{\cH}{\EuScript{H}}
\DeclareMathOperator{\cT}{\mathsf{T}}
\DeclareMathOperator{\vF}{\mathsf{v}}
\newcommand{\dd}{\mathrm{d}}

\DeclareMathOperator{\chr}{\mathrm{chr}}
\DeclareMathOperator{\cyc}{\mathrm{cyc}}
\DeclareMathOperator{\reg}{\mathrm{reg}}
\DeclareMathOperator{\Rrr}{\mathrm{R}}
\DeclareMathOperator{\uF}{\mathsf{u}}
\DeclareMathOperator{\wF}{\mathsf{w}}
\DeclareMathOperator{\Sp}{\mathrm{Sp}}
\DeclareMathOperator{\Uu}{\mathrm{U}}

\DeclareMathOperator{\HP}{\mathbb{H}\mathit{P}}

\DeclareMathOperator{\Pic}{\mathrm{Pic}}
\DeclareMathOperator{\ch}{\mathrm{ch}}
\DeclareMathOperator{\Bc}{\mathrm{Bc}}
\DeclareMathOperator{\Ff}{\EuScript{F}}

\makeatletter
  \DeclareSymbolFont{ntxletters}{OML}{ntxmi}{m}{it}
  \SetSymbolFont{ntxletters}{bold}{OML}{ntxmi}{b}{it}
  \re@DeclareMathSymbol{\leftharpoonup}{\mathrel}{ntxletters}{"28}
  \re@DeclareMathSymbol{\leftharpoondown}{\mathrel}{ntxletters}{"29}
  \re@DeclareMathSymbol{\rightharpoonup}{\mathrel}{ntxletters}{"2A}
  \re@DeclareMathSymbol{\rightharpoondown}{\mathrel}{ntxletters}{"2B}
  \re@DeclareMathSymbol{\triangleleft}{\mathbin}{ntxletters}{"2F}
  \re@DeclareMathSymbol{\triangleright}{\mathbin}{ntxletters}{"2E}
  \re@DeclareMathSymbol{\partial}{\mathord}{ntxletters}{"40}
  \re@DeclareMathSymbol{\flat}{\mathord}{ntxletters}{"5B}
  \re@DeclareMathSymbol{\natural}{\mathord}{ntxletters}{"5C}
  \re@DeclareMathSymbol{\star}{\mathbin}{ntxletters}{"3F}
  \re@DeclareMathSymbol{\smile}{\mathrel}{ntxletters}{"5E}
  \re@DeclareMathSymbol{\frown}{\mathrel}{ntxletters}{"5F}
  \re@DeclareMathSymbol{\sharp}{\mathord}{ntxletters}{"5D}
  \re@DeclareMathAccent{\vec}{\mathord}{ntxletters}{"7E}
  \DeclareSymbolFont{cmletters}{OML}{cmm}{m}{it}
  \SetSymbolFont{cmletters}{bold}{OML}{cmm}{b}{it}
  \re@DeclareMathSymbol{\wp}{\mathord}{cmletters}{"7D}
\makeatother

\renewcommand{\epsilon}{\varepsilon}

\DeclareFontFamily{U}{EBSUB}{}
\DeclareFontShape{U}{EBSUB}{m}{it}{<-> EBGaramond-Italic-tlf-lgr}{}
\DeclareFontFamily{U}{EBSUB}{}
\DeclareSymbolFont{EBSUB}{U}{EBSUB}{m}{it}

\DeclareMathSymbol{\mu}{\mathord}{EBSUB}{`m}
\DeclareMathSymbol{\varDelta}{\mathord}{EBSUB}{`D}
\DeclareMathSymbol{\varOmega}{\mathord}{EBSUB}{`W}

 \hypersetup{
    colorlinks,
    linkcolor={red!50!black},
    citecolor={blue!50!black},
    urlcolor={blue!80!black}
}

\usepackage{stackrel}
\usepackage[all]{xy}

\usepackage{bm}

\usepackage{indentfirst}

\newcommand{\pullbackcorner}[1][dr]{\save*!/#1-1.7pc/#1:(-1,1)@^{|-}\restore}

\usepackage{tocloft}

\title{\bf Two-Valued Groups, Chazy Equation,\\
Dubrovin--Frobenius Structures, and QYBE}

\date{}

\author{Victor Buchstaber, Mikhail Kornev, and Vladimir Rubtsov}

\begin{document}

\maketitle

\begin{abstract}
We show that the associativity condition of the universal symmetric $2$-algebraic 2-valued group defined by the Buchstaber polynomial admits several mutually equivalent interpretations from the viewpoints of the Chazy equation, Gauss--Manin connections, Dubrovin--Frobenius structures, and the quantum Yang--Baxter equation. These results place the universal 2-valued law in a unified framework linking geometry, algebraic topology, group theory, and mathematical physics.
\end{abstract}

{\small\tableofcontents}

\thispagestyle{empty}
\clearpage
\setcounter{page}{1}

\section{Introduction}

The aim of this paper is to present several mutually equivalent, but conceptually distinct, interpretations of the associativity condition for the universal symmetric $2$-algebraic 2-valued group. These interpretations come from the theory of quasimodular forms and the Chazy equation, the Gauss--Manin connection, Frobenius algebras and Dubrovin--Frobenius structures, and the quantum Yang--Baxter equation. Throughout the paper, the geometry of discriminant sets plays an important role; this theme was central to the work of Vladimir Igorevich Arnold and his school.

In \cite{Buchstaber90}, a three-parameter family of polynomials was introduced, later called the {\it Buchstaber polynomials}
\begin{equation}\label{B}B_{\boldsymbol{a}}(z; x, y) = (x + y + z - a_2xyz)^2 - 4(1 + a_3xyz)(xy + yz + xz + a_1xyz).\end{equation}
It defines the multiplication law in a two-valued group that is universal in the class of two-valued formal groups obtained by applying the modulus square construction to formal groups of the form
\begin{equation}\label{F_u_v}\Ff(u, v) = \frac{u^2\lambda_1(v)-v^2\lambda_1(u)}{u\lambda_2(v)-v\lambda_2(u)},\end{equation}
where $\lambda_1(u)$ and $\lambda_2(u)$ are formal power series over some $\Q$-algebra $A$.

The class of formal groups of the form \eqref{F_u_v} includes the formal groups defining the fundamental Hirzebruch genera: the Todd genus, the signature, the elliptic genus, and the Krichever genus.

Recently, in \cite{BK_disciminants}, the notions of algebraic $n$-valued monoids and groups were introduced. An {\it algebraic $n$-valued monoid} is an algebraic variety $X$ endowed with an associative $n$-valued multiplication (addition) defined by a rational morphism\footnote{That is, defined on some open subset $Y\subset X\times X$ in the Zariski topology by a morphism of algebraic varieties $\ast: Y\to \Sym^n(X)$.}
$$X\times X\to \Sym^n(X)$$
and a neutral element (zero) $e\in X$ such that $x\ast e = e\ast x = [x, x, ..., x]$ for every $x\in X$. An {\it algebraic $n$-valued group} is an algebraic $n$-valued monoid on $X$ equipped with a rational inversion morphism $\inv: X\to X$ for which, whenever $\inv(x)$ exists, the following conditions hold: $x\ast\inv(x)$ and $\inv(x)\ast x$ exist, and $e\in x\ast\inv(x)$, $e\in \inv(x)\ast x$.

In \cite{BK}, a class of symmetric $n$-algebraic $n$-valued groups $\G_{\Cc}(P_n)$ on $\Cc$ was introduced. They have neutral element $0$, inverse map $\inv(x) = (-1)^nx$, and multiplication
$$x\ast y = [z\mid P_n(z; x, y) = 0],$$
where $P_n(z; (-1)^nx, (-1)^ny)$ is a polynomial symmetric in $x,y,z$, of degree at most $n$ in each of the variables $x$, $y$, and $z$.

More precisely, the polynomial
\begin{equation}\label{polynom_P}P_n(z; x, y) = \sum\limits_{k = 0}^nF_{n-k}(x, y)z^{k},\end{equation}
where $F_{n-k}(x,y)\in\Cc[x, y]$ and
\begin{equation}\label{F_0(0, 0)=1}F_0(0, 0) = 1,\end{equation}
defines the operation $\ast$, which assigns to elements $x$ and $y$ of the complex numbers $\Cc$, such that $F_0(x, y)\neq 0$, the multiset $x\ast y$ of roots of the polynomial $P_n(z;x,y)$ with respect to the variable $z$.

Let us record the axioms of an $n$-valued algebraic group in this case.

From the neutral element axiom with $x = 0$, we obtain the equality
\begin{equation}\label{G(P)_neutral_element}\begin{aligned}P_n(z; 0, y) = F_0(0, y)(z-y)^n.\end{aligned}\end{equation}

The inverse element axiom says that there exists an algebraic map $\inv: \Cc\to \Cc$, such that for every $x\in\Cc$
\begin{equation}\label{G(P)_inverse_element}F_n(x, \inv(x)) = F_n(\inv(x), x) = 0.\end{equation}

Denote by $z_j = z_j(x, y)$, $j = 1, ..., n$, the roots of the polynomial $P_n(z; x,y)$. Then the associativity condition can be written as the identity of two polynomials in the variable $t$:
\begin{equation}\label{assoc_condition}\begin{aligned} \prod\limits_{j, k = 1}^n (t - z_j(z_k(x, y), z)) = \prod\limits_{j, k = 1}^n (t - z_j(x, z_k(y, z))). \end{aligned}\end{equation}
Observe that the coefficients of the powers of the variable $t$ in the polynomials on both sides of equality (\ref{assoc_condition}) are polynomials in $x$, $y$, and $z$.

For $n=2$, the polynomial $P_n(z; x, y)$ satisfying conditions \eqref{polynom_P}--\eqref{G(P)_inverse_element} takes the following form in the elementary symmetric functions $e_1,e_2$ and $e_3$:
\begin{equation}\label{P_2_universal_0}P_2(z; x, y) =  e_1^2 - 4 e_2 + k_2 e_3 + k_8 e_3^2 + k_4 e_1 e_3 + k_6 e_2 e_3.\end{equation}
In \cite[Theorem 6.3]{Buchstaber_Veselov19} and \cite[Theorem 1]{BK}, a complete classification of the groups $\G_{\Cc}(P_2)$ for $n=2$ was obtained. It was established that the associativity axiom \eqref{assoc_condition} for the law \eqref{P_2_universal_0} is equivalent to the relation:
\begin{equation}\label{assoc_condition_first} 4k_8 = k_4^2-k_6k_2. \end{equation}
The approach of \cite{Buchstaber_Veselov19} relies on the complete classification of all two-valued one-dimensional formal groups over $\Q$-algebras from \cite{Buchstaber75}. This classification uses classical results from the theory of generalized shift operators on coalgebras. In \cite{BK}, an explicit algebraic construction of the universal two-valued group in the class of groups $\G_{\Cc}(P_2)$ is given.

The polynomial $B_{\bm{a}}(z; x,y)$, see \eqref{B}, takes the following form in the elementary symmetric functions $e_1, e_2$, and $e_3$ with respect to the variables $x, y$, and $z$:
\begin{equation}\label{B_a_e_0}B_{\bm{a}}(z;x,y)=e_1^2-4e_2-4a_1e_3-2a_2 e_1 e_3-4a_3 e_2 e_3+(a_2^2-4a_1 a_3)e_3^2.\end{equation}
Comparing formulas \eqref{P_2_universal_0} and \eqref{B_a_e_0}, we see that one can set:
\begin{equation}\label{k_and_a_intro}k_2 = -4a_1,\quad k_4 = -2a_2,\quad k_6 = -4a_3, \quad k_8=a_2^2-4a_1a_3.\end{equation}

For any choice of complex parameters $\boldsymbol{a} = (a_1, a_2, a_3)$, the universal symmetric 2-algebraic two-valued group is realized in the form of the law \eqref{B_a_e_0}, which is obtained after the substitution $x\mapsto -1/x$, $y\mapsto-1/y$, $z\mapsto-1/z$ in the two-valued law $D_{\bm{a}}(z;x,y)$ of the coset construction of the group of points of the elliptic curve
$$\E = \{\eta^2 = \zeta^3+a_1\zeta^2+a_2\zeta+a_3\}$$
with respect to the involution $(\zeta, \eta)\mapsto (\zeta,-\eta)$, $\infty\mapsto\infty$. The polynomial $D_{\bm{a}}(z;x,y)$ is called the {\it generalized Kontsevich polynomial} in \cite{Kontsevich_type_polynomials}.

According to \cite[Theorem 4.7]{Kontsevich_type_polynomials}, the law (\ref{B}) for nonzero discriminant $\delta_{\bm{a}} = \Delta_t(t^3+a_1t^2+a_2t+a_3)$ extends from $\mathbb{C}$ to $\Cc\!P^1$, yielding a structure of a two-valued algebraic group $\G_{\CP^1}(B_{\bm{a}})$. In \cite{BK_Uspekhi, BK_disciminants}, it is shown that for $\delta_{\bm{a}} = 0$, the universal law \eqref{B_a_e_0} degenerates to a two-valued law realized by the coset construction from the group of points of a degenerate cubic.

Note that \cite[Theorem 2]{BK} gives a classification of symmetric 3-algebraic three-valued groups.

We now state the main results of the paper.

Consider the dynamical system in the phase space with coordinates $k_2$, $k_4$, and $k_6$:
\begin{equation}\label{dyn_sys_2} k_2' = \frac{\lambda}{2}k_4,\quad k_4' = \frac{3\lambda}{4}k_6,\quad k_6' = \frac{\lambda}{4}(k_4^2-k_2k_6)=\lambda k_8, \end{equation}
where $\lambda$ is a nonzero constant, $k_j=k_j(\tau)$, $\tau$ is a complex parameter on the upper half-plane $\Hh$ and $(-)' = 1/(2\pi i)\dd/\dd\tau$. There exists a one-parameter family $\G_{\Cc}(B_{\tau})$ of symmetric $2$-algebraic two-valued groups containing representatives of all isomorphism classes of symmetric $2$-algebraic two-valued groups arising from coset constructions of elliptic curves
\begin{equation}\label{cubic_2}y^2 = x^3-k_2/4\,x^2-k_4/2\,x-k_6/4\end{equation}
(it is assumed that the discriminant $\delta_{\bm{k}}$ of the polynomial on the right-hand side is nonzero). Within the family $\G_{\Cc}(B_{\tau})$, the associativity condition \eqref{assoc_condition_first} of the universal symmetric 2-algebraic two-valued group (see \eqref{P_2_universal}) is equivalent to the Chazy III equation
\begin{equation}\label{Chazy_equation_2}y''' = yy'' - \frac{3}{2}y'^2,\end{equation}
where $y=y(\tau)=-\lambda k_2/4$. The case of a nodal cubic corresponds to the constant family $k_2\equiv1, k_4\equiv k_6\equiv0$, and the case of a cuspidal cubic corresponds to the family $k_2=k_4=k_6=0$. This is the content of Theorem \ref{phase_flow}.

The normalization $\lambda=-1/\pi^2$ corresponds to the solution of equation \eqref{Chazy_equation_2} in the form of the quasimodular form $y(\tau) = E_2(\tau)$.

The space $\Sol$ of solutions of the Chazy equation carries a natural action of the group $\SL_2(\mathbb C)$. Under this action, each solution $y\in\Sol$ is sent to the solution
\begin{equation}\label{action_2}(y\big\|_{2} g)(\tau) := \frac{1}{(c\tau+d)^2}y\left(g\tau\right) - \frac{6c}{c\tau + d},\end{equation}
where the notation $g\tau$ denotes the fraction $(a\tau + b)/(c\tau + d)$. The space $\Sol$ decomposes into three orbits represented by the solutions $E_2(\tau)$, $0$, and $1$. Any one-parameter family $\G_{\Cc}(B_{\tau})$ of symmetric 2-algebraic two-valued groups is obtained from one of the three families of Theorem \ref{phase_flow} by the action of an element of the group $G=\SL_2(\Cc)$ in the sense of \eqref{action_2}. Thus, the $G$-equivalence classes of one-parameter families $\G_{\Cc}(B_{\tau})$ are in one-to-one correspondence with the $G$-orbits $G(E_2)$, $G(0)$, $G(1)$ of the solutions $y = E_2$, $y = 0$, and $y = 1$ of the Chazy equation \eqref{Chazy_equation_2}.

For a smooth family of elliptic curves $\pi:\E\to \cT$, we consider the Gauss--Manin connection on $\Hhh^1_{\dR}(\E/\cT)$. We follow the approach proposed by Katz and Oda in \cite{Katz_Oda68}. In an explicit basis of relative cohomology (using the Griffiths--Dwork reduction in Movasati's form \cite{Movasati2012}), the connection matrix is written down and the Ramanujan vector field is defined. The classical Ramanujan identities \eqref{Ramanujan} and the Chazy III equation are interpreted as certain horizontality conditions for the Ramanujan vector field, see Proposition \ref{prop:Ramanujan}.

Theorem \ref{Dubrovin-Frobenius_and_2-valued} states that the same condition \eqref{assoc_condition_first} is equivalent to the associativity of the three-dimensional Dubrovin--Frobenius structure corresponding to the potential
\begin{equation}\label{potential_F}F(t) = \frac{1}{2}(t^1)^2t^3+\frac{1}{2}t^1(t^2)^2 + f(t^2,t^3).\end{equation}

We show that the Casimir element of the Dubrovin--Frobenius structure with potential \eqref{potential_F} satisfies the quantum Yang--Baxter equation if and only if condition \eqref{assoc_condition_first} holds. This is a corollary of the surprising Theorem \ref{general_Yang_Baxter_and_associativity_condition}. In this theorem, the 4-parametric family $R = R(a,b,c,d)$ of $(9\times9)$-matrices \eqref{R_matrix_2} is presented. It is shown that the matrices $R$ define solutions of the QYBE if and only if $a^2 - d - bc = 0$.

Thus, we obtain the following scheme of equivalences:
{\small
\begin{equation}\label{diagram_of_equivalences}
\xymatrix@C=2em@R=1.5em{
& \txt{Horizontality\\ of the Ramanujan\\ vector field}\ar@{<=>}[d]  &\\
\txt{Chazy equation\ } \ar@{<=>}[r] & \txt{Associativity\\ of the two-valued group} \ar@{<=>}[r] \ar@{<=>}[d] & \txt{\ Associativity\\ \ of the Dubrovin--\\ \ Frobenius structure}\\
& \txt{QYBE for the\\ Dubrovin--Frobenius structure} &
}
\end{equation}
}

The structure of the paper is as follows. In Section \ref{Recollection_on_quasimodular_forms}, we recall the facts from the theory of quasimodular forms and elliptic curves needed for our purposes. In Section \ref{Universal_symmetric_group}, we discuss the law of the universal symmetric 2-algebraic two-valued group and its realization via a coset construction. In Section \ref{modulus_square_construction}, we recall the modulus square construction and show connections between the theory of elliptic functions and the theory of two-valued formal groups. In Section \ref{Curves_in_moduli_space}, we show the relation between the associativity equation and the Chazy III equation and establish a correspondence between the set of $\SL_2(\Cc)$-orbits in the space of solutions and the $\SL_2(\mathbb C)$-classes of one-parameter families of two-valued groups. In Section \ref{Ramanujan_System_and_Gauss_Manin_Connections}, we derive a geometric interpretation of the Ramanujan system in terms of the Gauss--Manin connection. In Section \ref{Frobenius_Structures}, we give explicit constructions of Frobenius algebras, Frobenius $n$-homomorphisms, Dubrovin--Frobenius structures, as well as the quantum Yang--Baxter equation and its $R$-matrix. The article ends with Section \ref{conclusion} where we present further directions and open problems.

\section{Basics of Quasimodular Forms and Elliptic Curves}\label{Recollection_on_quasimodular_forms}

To fix notation, let us recall some facts from the theory of quasimodular forms, following \cite{Zagier123}.

Let $\Hh$ be the upper half-plane, $\tau = u+iv\in\Hh$, and $\Gamma \subset \SL_2(\Z)$ a subgroup of the special linear group of $2\times 2$ matrices over the integers. The group $\SL_2(\Z)$ acts on $\Hh$ by M\"obius transformations:
$$\gamma\tau=\frac{a\tau+b}{c\tau+d}.$$
For every function $f:\Hh\to\Cc$ and integer weight $k$, the slash operator is defined by
$$(f|_k\gamma)(\tau) := (c\tau+d)^{-k}\,f(\gamma\tau).$$

Recall that a {\it holomorphic modular form of weight $k$ on $\Gamma$} is a holomorphic function $f:\Hh\to \Cc$ such that the following two conditions hold:
\begin{enumerate}[\bf (i)]
\item (Modular transformation law) $f|_k\gamma=f$ for every $\gamma\in\Gamma$.
\item (Holomorphy condition at the cusps) $f(u+iv) = O(e^{Cv})$ as $v\to\infty$ and $f(u+iv) = O(e^{C/v})$ as $v\to 0$ for some $C>0$. We denote by $\Hol_0(\Hh)$ the class of functions satisfying this condition.
\end{enumerate}

An important example of a modular form is the Eisenstein series, defined for each $k\geq4$ by
$$E_k(\tau) = \sum\limits_{\Gamma_1/\Gamma_{\infty}}1|_k\gamma = \frac{1}{2}\sum_{\substack{c,d\in\mathbb{Z}\\ \gcd(c,d)=1}}\frac{1}{(c\tau+d)^k}.$$

An {\it almost holomorphic modular form of weight $k$ and depth at most $p$ on $\Gamma$} is a real-analytic function $F:\Hh\to\Cc$ such that $F|_k\gamma=F$ for every $\gamma\in\Gamma$, and $F$ can be represented as a polynomial in $1/v$ with holomorphic coefficients, i.e.
\begin{equation}\label{almost_holomorphic_modular_form}F(\tau)=\sum_{r=0}^{p} f_r(\tau)\,v^{-r}\end{equation}
for some $p\geq 0$, where each function $f_r(\tau)$ is holomorphic on $\Hh$ and belongs to the class $\Hol_0(\Hhh)$. Denote by $\widehat{\Mm}^{\leq p}_k$ the space of such forms for the given $\Gamma$. Further, let
$$\widehat{\Mm}_\ast = \bigoplus\limits_k \widehat{\Mm}_k,\quad\widehat{\Mm}_k=\bigcup\limits_p\widehat{\Mm}^{\leq p}_k$$
be the graded and filtered ring of all almost holomorphic modular forms.

Let $F$ be an almost holomorphic modular form as in \eqref{almost_holomorphic_modular_form}. Then the function $f_0(\tau)$ is called a {\it quasimodular form of weight $k$ and depth at most $p$}. Equivalently, this is a function $f$ in the class $\Hol_0(\Hhh)$ such that, for fixed $\tau\in\Hhh$ and variable $\gamma =\left( \begin{matrix} a & b\\ c & d \end{matrix}\right)\in\Gamma$, the function $(f|_k\gamma)(\tau)$ is a polynomial of degree at most $p$ in the variable $c/(c\tau+d)$. Introduce the notation
$$\widetilde{\Mm}_\ast = \bigoplus\limits_{k}\widetilde{\Mm}_k,\quad \widetilde{\Mm}_k = \bigcup\limits_p \widetilde{\Mm}^{\leq p}_k,$$
where $\widetilde{\Mm}^{\leq p}_k$ denotes the space of quasimodular forms of weight $k$ and depth at most $p$ on $\Gamma$.

As is well known, the algebras $\widehat{\Mm}_\ast$ and $\widetilde{\Mm}_\ast$ are isomorphic.

\begin{prop}[\cite{Zagier123}]
Let $\Gamma$ be a non-cocompact discrete subgroup of the group $\SL_2(\R)$, and let $\phi$ be any quasimodular form of weight $2$ that is not modular $($for example, $\phi = E_2$ in the case $\Gamma\subset\SL_2(\Z)$). Denote by $D = (-)'=: \frac{1}{2\pi i}\frac{d}{d\tau} = q\frac{d}{dq}$ the Euler derivative. Then the following statements hold:
\begin{enumerate}[\bf (i)]
\item For $\Gamma = \SL_2(\Z)$, the algebra of modular forms $\Mm_\ast$ is isomorphic to the polynomial algebra $\Cc[E_4, E_6]$.

\item For $\Gamma = \SL_2(\Z)$, the algebra of almost holomorphic modular forms $\widehat{\Mm}_\ast$ is isomorphic to $\Cc[E_2^\ast, E_4, E_6]$. 

\item For $\Gamma = \SL_2(\Z)$, the graded algebra of quasimodular forms $\widetilde{\Mm}_\ast$ is isomorphic to $\Cc[E_2, E_4, E_6]$.

\item The space of quasimodular forms on $\Gamma$ is closed under differentiation: $D(\widetilde{\Mm}^{\leq p}_k)\subset \widetilde{\Mm}^{\leq p+1}_{k+2}$ for all $k,p\geq 0$. The Ramanujan identities hold:
\begin{equation}\label{Ramanujan}E_2'=\frac{E_2^2-E_4}{12},\quad E_4' = \frac{E_2E_4-E_6}{3},\quad E_6' = \frac{E_2E_6-E_4^2}{2}.\end{equation}

\item Every quasimodular form on $\Gamma$ is a polynomial in $\phi$ with modular coefficients:
$\widetilde{\Mm}^{\leq p}_k(\Gamma) = \bigoplus\limits_{r=0}^p\Mm_{k-2r}(\Gamma)\phi^r$ for all $p,k\geq 0$.

\item Every quasimodular form on $\Gamma$ can be written uniquely as a linear combination of $\phi$ and derivatives of modular forms: for all $p,k\geq 0$, one has
$$\widetilde{\Mm}^{\leq p}_k(\Gamma) = \begin{cases} \bigoplus\limits_{r=0}^pD^r(\Mm_{k-2r}(\Gamma)) & \text{if }p<k/2,\\ \bigoplus\limits_{r=0}^{k/2-1}D^r(\Mm_{k-2r}(\Gamma))\oplus\Cc\cdot D^{k/2-1}\phi, & \text{if }p\geq k/2 \end{cases}$$ 
\end{enumerate}
\end{prop}

Let us recall the connection between the theory of quasimodular forms and the theory of elliptic curves.

Every elliptic curve over $\Cc$ is isomorphic to a curve in Weierstrass form:
\begin{equation}\label{ellitpic_curve_equation}y^2 = 4x^3 - g_2x - g_3,\end{equation}
where
$$g_2(\tau) = \frac{4\pi^4}{3}E_4(\tau),\quad g_3(\tau) = \frac{8\pi^6}{27}E_6(\tau)$$
denote the classical Weierstrass invariants associated with the lattice $\Lambda = \Z+\tau\!\Z$ for some $\tau\in\Hh$.

The following fact is well known:

\begin{prop}[{\cite[Ch. 20, §20.5(i), Eq. (20.5.1)]{Handbook_of_functions}}]\label{alpha}
The following identity holds:
\begin{equation}\label{wp_and_E_2}\wp(z;\tau)= -\frac{\dd^2}{\dd z^2}\log\theta_1(z;\tau)\;-\;\frac{\pi^2}{3}E_2(\tau),\end{equation}
where $\theta_1(z; \tau)$ is the Jacobi theta function with quasi-periods $1$ and $\tau$, defined by
\begin{equation}\label{vartheta}\theta_1(z;\tau)=\sum_{n\in\mathbb Z}(-1)^n
e^{\pi i (n+\tfrac12)^2\tau}\,e^{2\pi i (n+\tfrac12)z}.\end{equation}
\end{prop}

Let us give an example of a remarkable family of elliptic curves containing representatives of all classes in the moduli space of elliptic curves.

\begin{prop}\label{a_and_E}
Let
\begin{equation}\label{E_a}y^2 = x^3 + a_1(\tau)x^2 + a_2(\tau)x + a_3(\tau)\end{equation}
be a one-parameter family of elliptic curves $\E_\tau$, given in the chart $\{z=1\}$ on $\CP^2$ with coordinates $x, y, z$, where
\begin{equation}\label{a_and_E_2} a_1(\tau) \;=\; -\pi^2E_2(\tau),\quad a_2(\tau) \;=\;4\pi^4E_2'(\tau),\quad a_3(\tau) \;=\; -\frac{8}{3}\pi^6E_2''(\tau)\end{equation}
and primes denote the scaled derivatives $(-)'=1/(2\pi i)\dd/\dd\tau$. Then the family $\E_\tau$ realizes representatives of all elements of the moduli space of elliptic curves over $\Cc$. 
\end{prop}

\begin{proof}

Rewrite the equation of the curve $\E$ in the form of a Weierstrass model $\widetilde{y}^2=\widetilde{x}^3-g_2\widetilde{x}-g_3$. To do this, consider a shift $\alpha$ such that equation \eqref{E_a} takes the form:
\begin{equation}\label{form_of_E}y^2=(x+\alpha)^3-\frac{g_2}{4}(x+\alpha)-\frac{g_3}{4}.\end{equation}
Then
\begin{equation}\label{E_g_2_and_g_3}a_1=3\alpha,\quad a_2=3\alpha^2-\frac{g_2}{4},\quad a_3=\alpha^3-\frac{g_2\alpha}{4}-\frac{g_3}{4}. \end{equation}
Let us check that one can take the following expressions for $\alpha$, $g_2$, and $g_3$:
\begin{equation}\label{alpha_g_2_g_3}\alpha(\tau) = -\frac{\pi^2}{3}E_2(\tau),\quad g_2(\tau) = \frac{4\pi^4}{3}E_4(\tau),\quad g_3(\tau) = \frac{8\pi^6}{27}E_6(\tau).\end{equation}
Here $g_2$ and $g_3$ denote the classical Weierstrass invariants associated with the lattice $\Lambda = \Z+\tau\!\Z$ for some $\tau\in\Hh$.

The Weierstrass function $\wp(z;\tau)$ satisfies the differential equation
\begin{equation}\label{Weierstrass}\left(\frac{\dd}{\dd z}\wp(z;\tau)\right)^2=4\wp^3(z;\tau)-g_2(\tau)\wp(z;\tau)-g_3(\tau).\end{equation}
As is well known, when $\tau$ ranges over the upper half-plane $\Hh$, the pair $(\wp(z;\tau), \wp'_z(z;\tau))$ gives parametrizations of representatives of all isomorphism classes of elliptic curves. By Proposition \ref{alpha}, one may take $x = -\dfrac{\dd^2}{\dd z^2}\log \theta_1(z;\tau)$. By successively eliminating $E_4$ and $E_6$ from the first two Ramanujan identities \eqref{Ramanujan}, we obtain the equalities:
\begin{equation}\label{E_4_E_6}E_4 = E_2^2-12E_2',\quad\quad E_6 =E_2^3-18E_2E_2'+36E_2''.\end{equation}
Substituting formulas \eqref{E_4_E_6} into \eqref{alpha_g_2_g_3}, and then formulas \eqref{alpha_g_2_g_3} into \eqref{E_g_2_and_g_3}, we obtain formulas \eqref{a_and_E_2}. Thus, the proposition is proved.
\end{proof}

\section{Universal Symmetric 2-Algebraic 2-Valued Group}\label{Universal_symmetric_group}

Let us recall the following (see details in \cite{BK_disciminants}).

\begin{defi}
{An algebraic $n$-valued monoid} is an algebraic variety $X$ equipped with an associative $n$-valued multiplication defined by a rational morphism $X\times X\to \Sym^n(X)$ with a neutral element (unit) $e\in X$, i.e., $x\ast e$, $e\ast x$ exist and the equalities $x\ast e = e\ast x = [x, x, ..., x]$ hold for any $x\in X$. 

{\it An algebraic $n$-valued group} is an algebraic $n$-valued monoid on $X$ equipped with a rational inversion morphism $\inv: X\to X$ for which, whenever $\inv(x)$ exists, the following conditions hold: $x\ast\inv(x)$ and $\inv(x)\ast x$ exist, and $e\in x\ast\inv(x)$, $e\in \inv(x)\ast x$. 

An algebraic $n$-valued monoid (or group) on $X$ is called {\it regular} if the $n$-valued multiplication $X\times X\to \Sym^n(X)$ is defined on the whole variety $X\times X$ (and the map $\inv$ is defined everywhere on $X$ in the case of a group).
\end{defi}

The coset construction for the group of points on cubic curves provides a source of nontrivial examples of algebraic $n$-valued monoids and groups (see \cite{BK_disciminants}).

In \cite{BK}, the following remarkable class of algebraic $n$-valued groups on $\Cc$ was introduced:

\begin{defi}\label{sym_nal_g_nval_group}
A {\it symmetric $n$-algebraic $n$-valued group on $\Cc$} is an algebraic $n$-valued group $\G_{\Cc}(P)$ with multiplication
$$x\ast y = [z\mid P(z; x, y) = 0],$$
defined by a symmetric polynomial $P(z;(-1)^nx,(-1)^ny)$ in which each variable has degree at most $n$, with identity element $0$ and inverse map $\inv(x) = (-1)^nx$.
\end{defi}

Every symmetric polynomial of degree at most $2$ in each of the variables $x$, $y$, and $z$, satisfying conditions \eqref{polynom_P}--\eqref{G(P)_inverse_element}, is of the form
\begin{equation}\label{P_2_universal}P_2(z; x, y) =  e_1^2 - 4 e_2 + k_2 e_3 + k_8 e_3^2 + k_4 e_1 e_3 + k_6 e_2 e_3,\end{equation}
where $k_2$, $k_4$, $k_6$, and $k_8$ denote (complex) parameters.

As was shown in \cite[Theorem 6.3]{Buchstaber_Veselov19} and \cite[Theorem 1]{BK}, the polynomial $P_2(z; x, y)$ defines on $\Cc$ the structure of a symmetric $2$-algebraic two-valued group with identity element $0$ and inverse map $\inv(x) = x$ if and only if the equality
\begin{equation}\label{assoc} 4k_8 = k_4^2 - k_6k_2 \end{equation}
holds. Thus, equality \eqref{assoc} is equivalent to the associativity condition
$$(x\ast y)\ast z = x\ast (y\ast z)$$
for the multiplication
$$x\ast y = [z\mid P_2(z; x, y) = 0].$$

The family of polynomials \eqref{P_2_universal} with coefficient ring
$$R_0 := \Z[k_2,k_4,k_6,k_8]/(4k_8 - k_4^2 + k_6k_2)$$
is universal (initial) in the class of all polynomials defining the structure of a symmetric $2$-algebraic two-valued group on a $\Z[1/2]$-module \cite{Buchstaber75, Buchstaber}. In the present paper, we consider only complex specializations of the parameters $k_2$, $k_4$, $k_6$, and $k_8$. Namely, if $Q$ is any other polynomial defining on $\Cc$ the structure of a symmetric $2$-algebraic two-valued group with coefficient ring $R'$, then there exists a unique homomorphism $r: R_0\to R'$ such that $Q = r(P_2)$, where $r(P_2)$ denotes the action of the homomorphism $r$ on the coefficients of the polynomial $P_2$.

Thus, the polynomial
\begin{equation}\label{P_2_group_universal}P_2(z; x, y) =  e_1^2 - 4 e_2 + k_2 e_3 + \frac{k_4^2 - k_6k_2}{4} e_3^2 + k_4 e_1 e_3 + k_6 e_2 e_3,\end{equation}
with independent complex parameters $k_2$, $k_4$, and $k_6$, is universal in the class of all polynomials defining structures of symmetric $2$-algebraic two-valued groups.

\begin{prop}[\cite{Kontsevich_type_polynomials, BK_disciminants}]\label{coset_construction_on_elliptic_curves}
The universal symmetric $2$-algebraic two-valued group $\G_{\Cc}(P_2)$ is obtained by the coset construction from the group of points of the curve \eqref{E_a} with respect to the involution $(x,y)\mapsto (x,-y)$. In terms of the parameters $a_1$, $a_2$, and $a_3$ of the curve $\E$, the corresponding two-valued multiplication law can be written in the form of the {\it Buchstaber polynomial}
\begin{equation}\label{B_a}
B_{\bm{a}}(z;x,y)=(x+y+z-a_2xyz)^2-4(1+a_3xyz)(xy+yz+xz+a_1xyz).
\end{equation}
\end{prop}

\begin{proof}

Let complex parameters $\alpha$, $g_2$, and $g_3$ be such that the equation of the curve $\E$ can be rewritten in the form
\begin{equation}\label{a1a2a3_system}\eta^2 = (\zeta + \alpha)^3 -\frac{g_2}{4}(\zeta + \alpha) - \frac{g_3}{4},\qquad
\begin{cases}
a_{1}= & 3\alpha,\\
a_{2}= & 3\alpha^{2}-g_{2}/4,\\
a_{3}= & \alpha^{3}-g_{2}\alpha/4-g_{3}/4.
\end{cases}\end{equation}
Recall that the addition law on this curve is given by the formulas (for distinct points of the curve $\E$):
\begin{equation}\label{sum_formula_x3}(\zeta_1, \eta_1)\oplus (\zeta_2, \eta_2) = (\zeta_3, \eta_3),\qquad\left\{\begin{aligned} \zeta_3 &= -\zeta_1 - \zeta_2 - 3\alpha + \left(\frac{\eta_1 - \eta_2}{\zeta_1 - \zeta_2}\right)^2,\\ \eta_3 &= (\zeta_1 - \zeta_3)\cfrac{\eta_1 - \eta_2}{\zeta_1 - \zeta_2} - \eta_1.\end{aligned}\right.\end{equation}
In the case when the summands coincide, equality \eqref{sum_formula_x3} should be understood in the sense of the limit as $\zeta_2\to \zeta_1$.

There is a branched double covering
\begin{equation}\label{branched_covering}\pi: \E\to \Cc\!P^1,\end{equation}
defined in the chart $\{z = 1\}$ by $\pi(\zeta, \eta) = \zeta$ and $\pi(\infty) = \infty$, with branch points at the roots of the polynomial $\zeta^3 + a_1\zeta^2 + a_2\zeta + a_3$ and at $\infty$. The fibers of the map $\pi$ are in bijective correspondence with the points of the orbit space $\E/\langle \sigma \rangle$ with respect to the involution
\begin{equation*}\begin{aligned}\label{involution_sigma}\sigma:(\zeta,\eta)&\mapsto (\zeta,-\eta),\\ \infty&\mapsto \infty.\end{aligned}\end{equation*}
Applying the coset construction \cite[Theorem 1]{Buchstaber} for the involution $\sigma$ on the group of points of the elliptic curve, we obtain a structure $\E_{\langle\sigma \rangle}$ of a coset algebraic $2$-valued group on $\Cc\!P^1$ with identity element at the point $\infty$:
\begin{equation}\label{x1_x2_2-valued}\zeta_1 \ast \zeta_2 = \left [-\zeta_1 - \zeta_2 - 3\alpha + \left(\frac{\eta_1 \pm \eta_2}{\zeta_1 - \zeta_2}\right)^2\right ].\end{equation}

The values of the expression {\normalfont{(\ref{x1_x2_2-valued})}} are roots of the quadratic polynomial $D(z; \zeta_1, \zeta_2)$ in the variable $z$ \cite[Proposition 7]{BK_disciminants}:
\begin{equation}\label{D_polynomial}D(z; \zeta_1, \zeta_2) = \Theta_0(\zeta_1, \zeta_2)z^2 + \Theta_1(\zeta_1, \zeta_2)z + \Theta_2(\zeta_1, \zeta_2),\end{equation}
where
$$\Theta_0=16(\zeta_1-\zeta_2)^2,$$
$$\Theta_1=8\bigl(2g_3+g_2(\zeta_1+\zeta_2+2\alpha)-4\bigl(\zeta_1\zeta_2(\zeta_1+\zeta_2)+6\zeta_1\zeta_2\alpha+3(\zeta_1+\zeta_2)\alpha^2+2\alpha^3\bigr)\bigr),$$
$$
\begin{aligned}
\Theta_2&=(g_2+4\zeta_1\zeta_2)^2+16g_2(\zeta_1+\zeta_2)\alpha
+24\bigl(g_2-4\zeta_1\zeta_2\bigr)\alpha^2\\
&\quad-64(\zeta_1+\zeta_2)\alpha^3-48\alpha^4
+16g_3(\zeta_1+\zeta_2+3\alpha).
\end{aligned}
$$

The polynomial $D_{\bm{a}}(z;\zeta_1,\zeta_2):=D(-z; -\zeta_1, -\zeta_2)$ was called the generalized Kontsevich polynomial in \cite{Kontsevich_type_polynomials}.

The substitution $x\mapsto -1/x$, $y\mapsto -1/y$, $z\mapsto -1/z$, together with formulas \eqref{a1a2a3_system}, transforms this polynomial into
$$B_{\boldsymbol{a}}(z; x, y) = (xyz)^2D_{\boldsymbol{a}}(-1/z; -1/x, -1/y)$$
and sends the group $\E_{\langle\sigma\rangle}$ to the group $\G_{\CP^1}(B_{\bm{a}})$.

\end{proof}

In terms of the elementary symmetric functions $e_1$, $e_2$, and $e_3$ in the variables $x$, $y$, and $z$, the polynomial $B_{\bm{a}}(z; x,y)$ takes the form
\begin{equation}\label{B_a_e}B_{\bm{a}}(z;x,y)=e_1^2-4e_2-4a_1e_3-2a_2 e_1 e_3-4a_3 e_2 e_3+(a_2^2-4a_1 a_3)e_3^2.\end{equation}

Comparing formulas \eqref{B_a_e} and \eqref{P_2_universal}, we see that
\begin{equation}\label{k_and_a}k_2 = -4a_1,\quad k_4 = -2a_2,\quad k_6 = -4a_3, \quad k_8=a_2^2-4a_1a_3.\end{equation}

Note that the family of coefficients $k_2$, $k_4$, $k_6$, and $k_8$ often appears in the theory of elliptic curves. Let $\E$ be an elliptic curve given by the equation
$$\eta^2+\alpha_1\zeta\eta+\alpha_3\eta = \zeta^3+\alpha_2\zeta^2+\alpha_4\zeta+\alpha_6.$$
The expressions
\begin{align}\label{b-invariants}b_2 = \alpha_1^2 &+ 4\alpha_2,\quad b_4 = \alpha_1\alpha_3 + 2\alpha_4, \quad b_6 = \alpha_3^2 + 4\alpha_6,\\ b_8&=\alpha_1^2\alpha_6-\alpha_1\alpha_3\alpha_4+4\alpha_2\alpha_6+\alpha_2\alpha_3^2-\alpha_4^2\nonumber \end{align}
are called in \cite[Chapter III, Section 3.1]{Cremona_92} ``auxiliary values''. In terms of these values, one writes the invariants $c_4$ and $c_6$, the discriminant $\Delta$, and the $j$-invariant of the elliptic curve $\E$:
$$c_4 = b_2^2-24b_4,\quad c_6 = -b_2^3 + 36b_2b_4-216b_6, \quad \Delta = -b_2^2b_8 - 8b_4^3 - 27b_6^2 + 9b_2b_4b_6,\quad j = c_4^3/\Delta.$$
Moreover, the relations
\begin{equation} 4b_8 = b_2b_6-b_4^2\quad \text{and}\quad 1728\Delta=c_4^3-c_6^2 \end{equation}
hold. It is easy to see that the notations in \eqref{b-invariants} and \eqref{k_and_a} are related as follows:
$$\begin{aligned}\alpha_1=\alpha_3=0,\quad &\alpha_2=-a_1,\quad \alpha_4=-a_2,\quad \alpha_6=-a_3,\\ b_2 &= k_2,\quad b_4=k_4,\quad b_6=k_6,\quad b_8=-k_8. \end{aligned}$$

\section{Modulus Square Construction and Buchstaber Genus}\label{modulus_square_construction}

As we saw, a complete description of the universal family of symmetric $2$-algebraic two-valued groups is obtained by purely algebraic means via the polynomial \eqref{P_2_universal} with free parameters $k_2$, $k_4$, and $k_6$. Proposition \ref{coset_construction_on_elliptic_curves} shows that this universal polynomial arises from a certain geometric construction based on elliptic curves. However, one may naturally ask how this geometric picture emerges from the general theory of formal two-valued groups developed by Buchstaber in \cite{Buchstaber75}. This question was answered in \cite{Kornev26}. In this section, following that work, we recall the modulus square construction and the Buchstaber genus. 

In \cite{Buchstaber_Novikov}, a modulus square construction and the corresponding two-valued formal group in the complex cobordism ring $\Omega_{\Uu}$ were introduced. The general theory of two-valued formal groups was developed in \cite{Buchstaber75}. For the definition, examples, and properties of $n$-valued groups, see the survey \cite{Buchstaber}, or recent works \cite{BK, BK_disciminants}. Let us recall the {\it modulus square construction}. 

Let $\zeta_1 = \pr_1^{\ast}(\zeta)$ and $\zeta_2 = \pr_2^{\ast}(\zeta)$ be quaternionic line bundles, where $\zeta$ is the universal quaternionic line bundle over $\HP^{\infty}$, and $\pr_j:\HP^{\infty}\times\HP^{\infty}\to \HP^{\infty}$ are two projections onto the Cartesian factors. The embedding $\Sss^1\hookrightarrow\Sp(1)$ of the multiplicative circle group $\Sss^1$ into the group of unit quaternions $\Sp(1)$ induces a map
\begin{equation}\label{iota_map}
\iota:\Cc\!P^{\infty}\to\HP^{\infty}
\end{equation}
between their classifying spaces. The pullback of the bundle $\zeta_j$ ($j = 1,2$) under the map $\iota$ decomposes as a sum $\eta_j\oplus\overline{\eta}_j$, where $\eta_j$ is the universal complex line bundle. We have the following pullback diagram:
$$
\xymatrix{
(\eta_1\oplus\overline{\eta}_1)\otimes_{\Cc}(\eta_2\oplus\overline{\eta}_2)\pullbackcorner\ar[rr]\ar[d] & & \zeta_1\otimes_{\Cc}\zeta_2\ar[d]\\
\CP^{\infty}\times\CP^{\infty}\ar[rr]^{\iota\times \iota} & & \HP^{\infty}\times\HP^{\infty}
}
$$

There is an isomorphism of complex vector bundles: $$(\eta_1\oplus\overline{\eta}_1)\otimes_{\Cc}(\eta_2\oplus\overline{\eta}_2) \cong (\eta_1\eta_2\oplus\overline{\eta}_1\overline{\eta}_2)\oplus(\eta_1\overline{\eta}_2\oplus\overline{\eta}_1\eta_2).$$ Note that the bundles $\xi_1:= \eta_1\eta_2\oplus\overline{\eta}_1\overline{\eta}_2$ and $\xi_2 :=\eta_1\overline{\eta}_2\oplus\overline{\eta}_1\eta_2$ admit quaternionic structures. 

Let $\Pic(M)$ be the Picard group of isomorphism classes $[\xi]$ of complex line bundles over a complex manifold $M$. Consider an involution $\sigma: [\xi]\mapsto[\overline{\xi}]$. By abuse of notation, we will write $\xi$ instead of $[\xi]$. The points of the corresponding orbit space $X:=\Pic(M)/\sigma$ are identified with unordered pairs $[\xi, \overline{\xi}]$. We get the 2-valued coset group (see \cite[Section 6]{Buchstaber} for the definition) with multiplication $$[\xi,\overline{\xi}]\ast[\eta,\overline{\eta}] = [[\xi\eta,\overline{\xi}\overline{\eta}], [\xi\overline{\eta}, \overline{\xi}\eta]]$$ and neutral element $[\mathbf{1}_{\Cc}, \mathbf{1}_{\Cc}]$, where $\mathbf{1}_{\Cc}$ denotes a class of a trivial complex line bundle.

Note that $\zeta_1\otimes_{\Cc}\zeta_2$ cannot admit a quaternionic structure (an exercise in representation theory). But the bundle $\zeta_1\otimes_{\Cc}\zeta_2\otimes_{\Cc}\zeta_3$ over $\HP^{\infty}\times\HP^{\infty}\times\HP^{\infty}$ does admit one. A natural isomorphism $$(\zeta_1\otimes_{\Cc}\zeta_2)\otimes_{\Cc}\zeta_3\cong \zeta_1\otimes_{\Cc}(\zeta_2\otimes_{\Cc}\zeta_3)$$ is equivalent to the associativity of the two-valued operation $\ast$. 

This two-valued group is a geometric realization of a more general construction: 

\begin{prop}\label{modulus_square_prop}
Let $G$ be an abelian group and $\sigma: g\mapsto -g$ an involution. Denote by $X$ the orbit space $G/\sigma$ with points $[g, -g]$. Then $X$ carries an involutive commutative two-valued coset group structure with operation
$$
[g, -g]\ast[h,-h] = [[g+h, -g-h], [g-h, -g+h]]
$$ with neutral element $[e, e]$, and $e$ is a unit.
\end{prop}

\begin{proof}
 This follows directly from the coset construction, see \cite[Section 6]{Buchstaber}.
\end{proof}

\begin{defi}
The two-valued group considered in Proposition \ref{modulus_square_prop} is called the {\it modulus square construction} for an abelian group $G$. 
\end{defi}

The modulus square construction has the following infinitesimal analogue. Let $F(u, v)$ be a commutative formal group over an algebra $A$ with logarithm $g(u)$ and exponential $f(u)$ (they are defined over $A\otimes \Q$). Let $x = u\overline{u}$, $y = v\overline{v}$. Then, by definition, a {\it two-valued formal group} $\G_F$ is given by the law \begin{equation}\label{two-valued_law}x\ast y = [F(u,v) F(\overline{u}, \overline{v}), F(u, \overline{v})F(\overline{u}, v)],\end{equation} neutral element $x = 0$ and inverse $\inv(x) = x$.

Define the following formal series: \begin{equation}\label{Theta_1_and_Theta_2_z_1_and_z_2}\begin{matrix}z_1 = F(u, v)F(\overline{u}, \overline{v}), & & z_2 = F(u,\overline{v})F(\overline{u}, v)\\ \Psi_1 = z_1 + z_2, & & \Psi_2 = z_1z_2. \end{matrix}\end{equation} One can check that these series belong to the ring $A[[x, y]]$, i.e. $\Psi_1 = \Psi_1(x, y)$ and $\Psi_2 = \Psi_2(x, y)$, see \cite[Lemma 2.21]{Buchstaber_Novikov}. Hence, we get the following

\begin{prop}
The two-valued formal group $\G_F$ is determined by the roots $z_1,z_2$ of a quadratic polynomial over the ring $A[[x,y]]$ $($see \eqref{Theta_1_and_Theta_2_z_1_and_z_2} for the notation$)$$:$ \begin{equation}\label{Theta_1_and_Theta_2_general}\begin{aligned}z^2 - \Psi_1(x, y)&z+\Psi_2(x, y) = 0,\\ x\ast y &= [z_1, z_2].\end{aligned}\end{equation}
\end{prop} 

A general definition of a multi-valued formal group was given in \cite[Section 1]{Buchstaber75}.

\begin{defi}
In the notation of \eqref{Theta_1_and_Theta_2_z_1_and_z_2}, the series \begin{equation}\label{logarithm_of_two-vaued_group_2}B(x) = g(u)g(\overline{u}) = -g(u)^2\end{equation} is called a {\it logarithm} of the two-valued formal group $\G_F$ with the multiplication \eqref{two-valued_law}.
\end{defi}

The series $B(x)$ in \cite{Buchstaber_Novikov} was called the logarithm for the following reason: 

\begin{prop}
In the notation of \eqref{Theta_1_and_Theta_2_z_1_and_z_2}$\ and\ \eqref{logarithm_of_two-vaued_group_2}$$:$
$$
z_{1,2} = B^{-1}\left(\left(\sqrt{B\left(x\right)} \pm \sqrt{B\left(y\right)}\right)^2\right).
$$
\end{prop}

\begin{proof}
We have
$$
B^{-1}((\sqrt{B(x)} \pm \sqrt{B(y)})^2) = B^{-1}(-(g(u)\pm g(v))^2).
$$
It is enough to check that
$$
B(z_{1,2}) = -(g(u)\pm g(v))^2.
$$
Consider the case of $z_1$ (the case of $z_2$ is completely analogous). We have
$$
B(z_1) = B(V\overline{V}) = -g(V)^2 = -(gg^{-1}(g(u) + g(v)))^2 = -(g(u) + g(v))^2,
$$
where $V = g^{-1}(g(u) + g(v))$. This is exactly what was required.
\end{proof}

Recall the topological applications of the group $\G_F$ developed in \cite{Buchstaber_Novikov}. Consider the formal group $F(u, v) = F_{\Uu}(u, v)$ of geometric complex cobordisms over the ring $\Omega_{\Uu} = \Uu^{\ast}(\pt)$ with the logarithm $g(u) = g_{\Uu}(u)$ and the exponential $f(u) = f_{\Uu}(u)$ (see \cite[Appendix 1]{Novikov67}). 

Introduce classes $z_1, z_2\in \Uu^4(\CP^{\infty}\times\CP^{\infty})$ as Chern classes in $\Uu$-theory:
$$
\begin{aligned}
z_1 &:=  c_2^{\Uu}(\eta_1\eta_2\oplus \overline{\eta}_1\overline{\eta}_2) = c_1^{\Uu}(\eta_1\eta_2)c_1^{\Uu}(\overline{\eta}_1\overline{\eta}_2) = F_{\Uu}(u, v)F_{\Uu}(\overline{u}, \overline{v}),\\
z_2 &:=  c_2^{\Uu}(\eta_1\overline{\eta}_2\oplus \overline{\eta}_1\eta_2) = c_1^{\Uu}(\eta_1\overline{\eta}_2)c_1^{\Uu}(\overline{\eta}_1{\eta}_2) = F_{\Uu}(u,\overline{v})F_{\Uu}(\overline{u}, v),
\end{aligned}
$$
where $u = c_1^{\Uu}(\eta_1)$, $v = c_1^{\Uu}(\eta_2)$, $\overline{u} = c_1^{\Uu}(\overline{\eta}_1)$, $\overline{v} = c_1^{\Uu}(\overline{\eta}_2)$.

Let
$$
x = \iota^{\ast}p_1^{\Sp}(\zeta_1) = u\overline{u}
\qquad \text{and} \qquad
y = \iota^{\ast}p_1^{\Sp}(\zeta_2) = v\overline{v}\in \Sp^4(\CP^{\infty}),
$$ where $p_1^{\Sp}(\zeta_j)\in\Sp^{4}(\HP^{\infty})$ denotes the Borel class in $\Sp$-theory, and $\iota$ is the map \eqref{iota_map}. 
Define series:
$$
\begin{aligned}
\Psi_1 :&=p_1^{\Sp}( \iota^{\ast}(\zeta_1\otimes_{\Cc}\zeta_2)) = z_1 + z_2,\\
\Psi_2 :&=\nonumber p_2^{\Sp}(\iota^{\ast}(\zeta_1\otimes_{\Cc}\zeta_2))  =z_1z_2.
\end{aligned}
$$

Let $z$ be a generator of $\Hhh^4(\HP^{\infty},\Z)$ such that $\iota^{\ast}z=-t^2$, where $t$ is a generator of $H^{2}(\CP^{\infty}, \Z)$. Let $\zeta$ be a universal quaternionic bundle over $\HP^{\infty}$, and $\eta$ be a universal complex bundle over $\CP^{\infty}$. As we already know, $x = \iota^{\ast}c_2^{\Uu}(\zeta) = c_1^{\Uu}(\eta)c_1^{\Uu}(\overline{\eta}) = u\overline{u}$. By abuse of notation, sometimes we will drop the map $\iota^{\ast}$. Then \cite[page 93]{Buchstaber_Novikov}, $\ch_{\Uu}(g(u)) = t$ and $\ch_{\Uu}(B(x)) = -t^2 = z$, where $\ch_{\Uu}(x)$ is a Chern–Dold character $$\ch_{\Uu}:\Uu^{\ast}(\HP^{\infty})\to  \Hhh^{\ast}(\HP^{\infty}, \Omega_{\Uu}\otimes\Q)\cong\Omega_{\Uu}\otimes\Q[z],$$ which was introduced in \cite{Buchstaber70}. Thus, the inverse power series $B^{-1}(z)$ called {\it exponential} coincides with the Chern character: $$B^{-1}(x) = \ch_{\Uu}(z)\in \Hhh^{\ast}(\HP^{\infty}, \Omega_{\Uu}\otimes\Q).$$ 

All the 2-valued formal groups of the form \eqref{Theta_1_and_Theta_2_general} were classified by Buchstaber in \cite{Buchstaber75}. Recall one of the main results of that paper:

 \begin{theorem}[{\cite[Theorem 6.4]{Buchstaber75}}]\label{two_valued_group_differential_equation}
Let
\begin{equation}\label{general_two-valued_law}
x\ast y = \{z\mid z^2 - \Psi_1(x, y)z + \Psi_2(x, y) = 0\}
\end{equation}
be an arbitrary two-valued formal group $\G(R)$ in formal power series over an arbitrary $\mathbb Q$-algebra $R$. Let $B(x)$ be its logarithm. Then $B(x)$ satisfies the differential equation
$$
\frac{1}{2}\phi_1(x)B'(x) + \frac{1}{8}\phi_2(x)B''(x) = 1
$$
with the initial condition $B(0)=0$, where
$$
\begin{aligned}
\phi_1(x) &= \left.\frac{\partial\Psi_1(x,y)}{\partial y}\right|_{y=0},\quad
\phi_2(x) = \left.\frac{\partial\sigma(x,y)}{\partial y}\right|_{y=0},\quad
\sigma(x,y) = \Psi_1^2 - 4\Psi_2.
\end{aligned}
$$
If $\G(R)$ is of the first type, that is, $\Psi_2(x,y)\equiv (x-y)^2 \mod \deg 3,$ then the series $B(x)$ defines a strong isomorphism of this two-valued formal group with the elementary two-valued formal group defined by the polynomial
$$
z^2-2(x+y)z+(x-y)^2.
$$
Moreover, for the logarithm of $\G(R)$ we have:
\begin{equation}\label{B_of_x_and_phi_2}
B(x) = \left(\int\limits_{0}^{\sqrt{x}} \frac{\mathrm{d} t}{\sqrt{\phi(t^2)}} \right)^2,
\quad
\phi_2(x) = 8\int\limits_{0}^{x}\phi_1(t)\,\mathrm{d} t,
\end{equation}
where $\phi(t)=\phi_2(t)/(16t)$ and $\phi(0)=1$.
\end{theorem}

As a corollary, we get the following

\begin{theorem}\label{Buchstaber_genus_theorem}
\leavevmode
\begin{enumerate}[{\bf (i)}]

\item Let
\begin{equation}\label{Integral_I_formula}
\begin{aligned}
u := I(x)=\int\limits_{0}^{\sqrt{x}}\frac{\dd t}{\sqrt{1+a_1t^2+a_2t^4+a_3t^6}},\\
g_2 = 4\left (\frac{a_1^2}{3}-a_2 \right ), \qquad
g_3=4\left(\frac{a_1a_2}{3}-\frac{2a_1^3}{27}-a_3 \right).
\end{aligned}
\end{equation}
Then
$$
x(u) = \frac{1}{\wp(u;g_2,g_3)-a_1/3}.
$$
$($If $g_2^3-27g_3^2 = 0$, then the function $
\wp(u; g_2, g_3) = -\frac{\dd^2}{\dd u^2}\log\sigma(u;g_2,g_3)$
corresponds to a degeneration of the Weierstrass $\sigma$-function$)$.

\item The logarithm and exponential of the formal two-valued group $\G_{\bm{a}}$ are given by the following formal series over the algebra $\Q\left[a_1, a_2,a_3\right]$:
$$
B(x)= I^2(x) =
\left(\int\limits_{0}^{\sqrt{x}}\frac{\dd t}{\sqrt{1+a_1t^2+a_2t^4+a_3t^6}}\right)^2, \qquad
B^{-1}(x)=\frac{1}{\wp(\sqrt{x};g_2,g_3)-a_1/3}.
$$

\item Let $Q(t) := 1 - a_1 t^2 + a_2 t^4 - a_3 t^6$. Then the law $F_{\Bc}(u, v)$, its exponential $f_{\Bc}(u)$, and its logarithm $g_{\Bc}(u)$ have the form:
$$
F_{\Bc}(u,v)= \left[\left(\frac{u\sqrt{Q(v)}-v\sqrt{Q(u)}}{u^2-v^2}\right)^2+a_3u^2v^2\right]^{-1/2},
$$
\[
f_{\Bc}(u)=\frac{1}{\sqrt{\wp(u;g_2,-g_3)+\frac{a_1}{3}}},\qquad
g_{\Bc}(u) = \int\limits_{0}^{u}\frac{\dd t}{\sqrt{1-a_1t^2+a_2t^4-a_3t^6}}.
\]

\item The formal group $F_{\Bc}(u, v)$ is universal in the class of all single-valued formal groups
with the condition $\overline{u} = -u$, for which the modulus square construction yields the two-valued formal group $\G_{\bm{a}}$ with the two-valued law \eqref{B}.

\item  The intersection of the classes $\Ff(u, v)$ and $F_{\Bc}(u,v)$ coincides with the Ochanine genus. This is characterized by the conditions
$\lambda_1(u)\equiv 1$ for $\Ff(u,v)$ and $a_3=0$ for $F_{\Bc}(u,v)$.

\end{enumerate}
\end{theorem}

The corresponding Hirzebruch genus is known as the Buchstaber genus. The corresponding class of formal group laws $F_{\Bc}(u, v)$ is related to, but does not coincide with, the family of formal group laws associated with the Krichever genus. This genus gives examples not arising from Hirzebruch’s elliptic genera of level $n$. Its values $\Bc(\Theta^n)$, $\Bc(\CP^n)$ on smooth theta divisors and complex projective spaces can be found in \cite{Kornev26}.

\section{Curves in the Moduli Space of 2-Valued Groups}\label{Curves_in_moduli_space}

Consider the dynamical system in the phase space with coordinates $k_2,k_4,$ and $k_6$:
\begin{equation}\label{dyn_sys} k_2' = \frac{\lambda}{2}k_4,\quad k_4' = \frac{3\lambda}{4}k_6,\quad k_6' = \frac{\lambda}{4}(k_4^2-k_2k_6)=\lambda k_8, \end{equation}
where $\lambda$ is a nonzero constant, $k_j=k_j(\tau)$, and $(-)' = 1/(2\pi i)\dd/\dd\tau$.

In this section, we show that the trajectories of the dynamical system \eqref{dyn_sys} satisfy the following properties.

\subsection{The Associativity Condition and the Chazy Equation}

The main result of this section is the following.

\begin{theorem}\label{phase_flow}
Let
\[
B_{\tau}(z; x, y) =  e_1^2 - 4 e_2 + k_2(\tau) e_3 + \frac{k_4^2(\tau) - k_6(\tau)k_2(\tau)}{4} e_3^2 + k_4(\tau) e_1 e_3 + k_6(\tau) e_2 e_3,
\]
be a one-parameter family of Buchstaber polynomials with variables $k_2,k_4,k_6$ satisfying the dynamical system \eqref{dyn_sys}. Then the following statements hold:
\begin{enumerate}[\bf (i)]
\item There exists a one-parameter family $\G_{\Cc}(B_{\tau})$ of symmetric $2$-algebraic two-valued groups containing representatives of all isomorphism classes of symmetric $2$-algebraic two-valued groups corresponding to coset constructions of elliptic curves
\begin{equation}\label{cubic}y^2 = x^3-k_2/4\,x^2-k_4/2\,x-k_6/4\end{equation}
(the discriminant $\delta_{\bm{k}}$ of the polynomial on the right-hand side is nonzero).

\item Moreover, within the family $\G_{\Cc}(B_{\tau})$, the associativity condition
$4k_8 = k_4^2 - k_6k_2$
for the universal symmetric $2$-algebraic two-valued group (see \eqref{P_2_universal}) is equivalent to the Chazy equation
\begin{equation}\label{Chazy_equation}y''' = yy'' - \frac{3}{2}y'^2,\end{equation}
where $y =y(\tau) = -\lambda k_2/4$ and $(-)' = 1/(2\pi i)\dd/\dd\tau$.

\item The case of the nodal cubic corresponds to the constant family $k_2\equiv1$, $k_4\equiv k_6\equiv0$, while the case of the cuspidal cubic corresponds to the family $k_2=k_4=k_6=0$.
\end{enumerate}
\end{theorem}

\begin{proof}
Part {\bf (i)} is obtained analogously to Proposition \ref{a_and_E}.

Let us prove part {\bf (ii)}. From \eqref{dyn_sys} we know that, on the one hand,
\begin{equation}\label{k_8_one_side}k_8 = \frac{1}{\lambda}k_6' = \frac{4}{3\lambda^2}k_4'' = \frac{8}{3\lambda^3}k_2''' = -\frac{32}{3\lambda^4}y'''.\end{equation}
On the other hand,
\begin{equation}\label{k_8_the_other_side}k_8 = \frac{1}{4}(k_4^2-k_2k_6) = \frac{1}{\lambda^2}(k_2')^2- \frac{2}{3\lambda^2}k_2k_2'' = \frac{16}{\lambda^4}\left ( (y')^2 -\frac{2}{3}yy''\right ).\end{equation}
Equating \eqref{k_8_one_side} and \eqref{k_8_the_other_side}, we obtain \eqref{Chazy_equation}.

Part {\bf (iii)} follows by direct substitution of the parameters into \eqref{dyn_sys} and \eqref{cubic}.
\end{proof}

As is well known (see, for example, \cite{Zagier123}), the quasimodular form $E_2$ satisfies the Chazy equation \eqref{Chazy_equation}. Thus, Proposition \ref{a_and_E} is a consequence of Theorem \ref{phase_flow} for $\lambda = -1/\pi^2$.

\subsection{M\"obius Orbits of Solutions to the Chazy Equation}

The solution space $\Sol$ of the Chazy equation \eqref{Chazy_equation} carries a natural action of the group $\SL_2(\Cc)$. Under this action, each solution $y\in\Sol$ is sent to the solution
\begin{equation}\label{action}(y\big\|_{2} g)(\tau) := \frac{1}{(c\tau+d)^2}y\left(g\tau\right) - \frac{6c}{c\tau + d},\end{equation}
where, as usual, $g\tau$ denotes the fraction $(a\tau + b)/(c\tau + d)$.

Although the following result was already known to Chazy \cite{Chazy}, we were unable to find an explicit self-contained proof in the literature. For this reason, we provide it.

\begin{theorem}
The space $\Sol$ of solutions to the Chazy equation \eqref{Chazy_equation} decomposes into three orbits
$$\Sol = G(E_2) \sqcup G(0)\sqcup G(1)$$
under the action of the group $G = \SL_2(\Cc)$, namely the orbit of the Eisenstein series $y_0 = E_2$, the zero solution $y_1 = 0$, and the constant solution $y_2=1$.
\end{theorem}

\begin{proof}
Under the substitution $y\mapsto y/2$, equation \eqref{Chazy_equation} takes the form
\begin{equation}\label{original_Chazy_equation} y''' = 2yy'' - 3(y')^2. \end{equation}
We will assume that $(-)'$ denotes the derivative $\dd/\dd\tau$ (without the factor $1/(2\pi i)$). The solutions of the Chazy equations with respect to the derivatives $\dd/\dd\tau$ and $1/(2\pi i)\dd/\dd\tau$ differ only by the factor $2\pi i$.

Let $y = y(\tau)$ be a solution of the Chazy equation \eqref{original_Chazy_equation}. Introduce the auxiliary functions
$$Q(\tau):=y^2-6y',\qquad
R(\tau):=y^3-9yy'+9y'', \qquad \Delta(\tau):=Q^{3}-R^{2}.$$

Then
\begin{equation}\label{helpful_derivatioves}Q'=\frac23(yQ-R),\qquad R'=yR-Q^2, \qquad \Delta' = 2y\Delta.\end{equation}

\textbf{\textit{Case 1.}} $\Delta\not\equiv 0$. Then one can define the function
\begin{equation}\label{t(x)}t(\tau):=-\frac{R^2}{\Delta}.\end{equation}
A direct computation using formulas \eqref{helpful_derivatioves} shows that the function $t = t(\tau)$ satisfies the following Schwarz equation
\begin{equation}\label{Schwarzian_equation}
\{t;\tau\}=V\left(\frac{1}{2},0, \frac{1}{3}; t, \tau\right)(t')^2\;
\end{equation}
with potential
$$V(\alpha,\beta,\gamma; t, \tau) = \frac{1-\alpha^2}{t^2}+\frac{1-\beta^2}{(t-1)^2}+\frac{\alpha^2+\beta^2-\gamma^2-1}{t(t-1)},$$
where
$$\{t; \tau\} = \frac{t'''}{t'}-\frac{3}{2}\left(\frac{t''}{t'}\right)^2$$
is the Schwarzian derivative.

It is known (see \cite[Section 2.2]{Chakravarty} and the references therein) that the general solution of the Schwarz equation
$$\{t,\tau\}=V(\alpha,\beta,\gamma; t, \tau)(t')^2$$
for $\alpha+\beta+\gamma<1$ is a function $t(\tau)$ whose inverse is the Schwarz function $S(\alpha, \beta, \gamma; t)$. The function $S(\alpha, \beta, \gamma; t)$ maps the upper half-plane $\Hh$ to a triangle with angles $\pi\alpha$, $\pi\beta$, and $\pi\gamma$, and $S(\alpha, \beta, \gamma; t)$ is represented as a quotient $\phi(t)/\psi(t)$, where $\phi$ and $\psi$ are some linearly independent solutions of the Gauss hypergeometric equation
\begin{equation}\label{Gauss_hypergeometric_equation}
t(1-t)\,z''(t)+\bigl(c-(a+b+1)t\bigr)\,z'(t) - ab\,z(t) = 0
\end{equation}
with parameters $(a,b,c)$ related to $(\alpha,\beta,\gamma)$ by the equalities
$$(1-c,\ c-a-b,\ a-b)=(\alpha,\beta,\gamma).$$

Thus, in the case when $\Delta = Q^3-R^2$ is not identically zero, the function $\tau(t)$ is represented as the quotient $\phi(t)/\psi(t)$, where $\phi$ and $\psi$ are linearly independent solutions of equation \eqref{Gauss_hypergeometric_equation} for $(a,b,c) = (1/12, 1/12, 1/2)$.

A direct computation shows that
\begin{equation}\label{Delta_of_tau} \Delta = -\frac{(t')^6}{64t^3(1-t)^4}. \end{equation}
Moreover,
\begin{equation}\label{t_prime_of_tau}t'(\tau) =\frac{\dd t}{\dd\tau} = 1/\left(\frac{\phi(t)}{\psi(t)}\right)'= \frac{\psi^2(t)}{W(t)},\end{equation}
where $W(t) = \phi\psi'-\phi'\psi$ is the Wronskian of the solutions $\phi = \phi(t)$ and $\psi = \psi(t)$ of the Gauss hypergeometric equation. As is well known, the Wronskian $W$ of any second-order linear differential equation $az''+bz'+cz=0$ satisfies the differential equation
$$W'=-\frac{b}{a}W.$$
In our case,
\begin{equation}\label{W(t)}W(t) = Ct^{-1/2}(1-t)^{-2/3}\end{equation}
for some constant $C$. Substituting \eqref{W(t)} and \eqref{t_prime_of_tau} into \eqref{Delta_of_tau}, we obtain
$$\Delta = -\frac{1}{64C^6}\psi^{12}.$$
From \eqref{helpful_derivatioves} we get
$$y=\frac{1}{2}\frac{\Delta'}{\Delta}= 6\,\frac{1}{\psi}\frac{\dd\psi}{\dd\tau}=6\frac{\dd}{\dd\tau}\ln\psi.$$

Thus, every solution of the Chazy equation for which $\Delta\not\equiv 0$ is parametrized as follows:
$$\tau(t) = \cfrac{\phi(t)}{\psi(t)},\qquad y(\tau) = 6\frac{\dd}{\dd\tau}\ln \psi(t(\tau)).$$

Suppose we have another solution $\widetilde{\tau}(t) = \widetilde{\phi}(t)/\widetilde{\psi}(t)$, where
$\widetilde{\phi} = a\phi + b\psi$ and $\widetilde{\psi} = c\phi + d\psi$.
Without loss of generality, we may assume that
$\left(\begin{matrix} a & b\\ c & d \end{matrix}\right )\in\SL_2(\Cc)$.
Then, using the facts that $\phi = \tau\psi$ and
$\widetilde{\tau} = \frac{a\tau+b}{c\tau + d}$, we obtain
$$\widetilde{y}(\widetilde{\tau}) = 6\frac{\dd}{\dd\widetilde{\tau}}\ln\widetilde{\psi} = 6\frac{\dd\tau}{\dd\widetilde{\tau}}\frac{\dd}{\dd\tau}\ln\widetilde{\psi}= (c\tau+d)^2y + 6c(c\tau +d).$$
Thus,
$$\widetilde{y}(\tau) = \frac{1}{(c\tau - a)^2}y\left ( \frac{\tau d-b}{-c\tau+a}\right ) + \frac{6c}{-c\tau+a}.$$
Therefore, in this case the family of solutions coincides with the orbit $G(E_2/2)$ of the quasimodular form $E_2(\tau)/2$.

\textbf{\textit{Case 2.}} $\Delta\equiv 0$. Let $Q=u^2$ and $R = u^3$. Put $w:=y-u$. Then, from the equalities $Q = y^2-6y'$ and $Q'=\frac{2}{3}u^2(y-u)$, it follows that $w'=w^2/6$. Solving this differential equation, we obtain the following general form of the solution in this case:
$$y(\tau)= \frac{A}{(c\tau+d)^2} - \frac{6c}{c\tau+d},$$
where $A,c,$ and $d$ are arbitrary constants with $|c|+|d|\neq 0$. This family coincides with the orbit $G(0)$ of the zero solution when $A=0$, and with the orbit $G(1)$ of the constant solution when $A\neq 0$.

Thus, the theorem is proved.
\end{proof}

Note that, as observed in \cite{Buchstaber_Leikin_Pavlov03}, the monoid
$$\widetilde{G} = \left\{\left.\left(\begin{matrix} a & b\\ c & d\end{matrix}\right)\ \right| |c|+|d|>0 \right\}$$
of $(2\times 2)$-matrices acts on the space of solutions of the Chazy equation by the rule \eqref{action}. Under this action, nondegenerate solutions may pass to degenerate ones.

We are now ready to formulate the second main result of this section:

\begin{theorem}\label{modular_uniqeness}
Any one-parameter family $\G_{\Cc}(B_{\tau})$ of symmetric $2$-algebraic two-valued groups defined by the flow of the dynamical system \eqref{dyn_sys} is obtained from one of the three flow families of Theorem \ref{phase_flow} by the action of an element of the group $G=\SL_2(\Cc)$ in the sense of \eqref{action}. Thus, the classes of $G$-equivalence of one-parameter families $\G_{\Cc}(B_{\tau})$ are in one-to-one correspondence with the $G$-orbits $G(E_2)$, $G(0)$, and $G(1)$ of the solutions $y = E_2$, $y = 0$, and $y = 1$ of the Chazy equation \eqref{Chazy_equation}.
\end{theorem}

\section{Ramanujan System and Gauss–Manin Connections}\label{Ramanujan_System_and_Gauss_Manin_Connections}

In this section, the associativity condition for the universal $2$-algebraic two-valued group is described in the context of the Gauss--Manin connection.

\subsection{Basics of Gauss–Manin Connections}\label{recollection_on_Gauss-Manin_connection}

Let us recall the construction first introduced by Manin in \cite{Manin58} and later called the Gauss--Manin connection, following \cite{Katz_Oda68}. Let $\pi:\E\to\cT$ be a smooth proper morphism between complex algebraic varieties. Consider the sheaf
\[
\Omega^1_{\E/\cT} \;:=\; \Omega^1_{\E}\big/\pi^*\Omega^1_{\cT}
\]
of relative $1$-forms and the relative de Rham complex
$$\Omega^\bullet_{\E/\cT}:\quad
\Oo_{\E}\xrightarrow{\dd_{\E/\cT}}
\Omega^1_{\E/\cT}\xrightarrow{\dd_{\E/\cT}}
\Omega^2_{\E/\cT}\to\cdots$$
of $\Oo_{\E}$-modules.

Let
\[
\cH^n:=\Hhh^n_{\dR}(\E /\cT)=\Rrr^n\!\pi_{\ast}(\Omega^\bullet_{\E/\mathcal T})
\]
denote the sheaf of relative de Rham cohomology. Since the morphism $\pi$ is smooth and proper, $\cH^n$ is a locally free $\Oo_{\cT}$-module, i.e. a vector bundle. The base change isomorphism with respect to the morphism $i:\Spec k(t)\to \cT$ shows that the fiber of this bundle over a closed point $t\in\cT$ is the vector space
$$\cH^n\otimes_{\Oo_{\cT}} k(t) \;\cong\; H^n_{\mathrm{dR}}(\mathcal E_t),$$
the $n$th de Rham cohomology of the variety
\[
\E_t:=\E\times_{\cT}\Spec k(t)
\]
in the fiber over the point $t\in\cT$.

By the definition of relative differentials, one has a short exact sequence of sheaves:
\begin{equation}\label{short_exact_seq_of_relative_differentials}0\to \pi^*\Omega^1_{\cT}\to \Omega^1_{\E}\to \Omega^1_{\E/\cT}\to0. \end{equation}
This sequence induces an increasing filtration on the absolute de Rham complex $\Omega_{\E}^{\bullet}$: for all nonnegative integers $n$ and $p$, define $F^p\Omega_{\E}^n$ to be the subsheaf locally generated by differential forms each of which contains at least $p$ factors from the space $\pi^\ast\Omega_{\cT}^1$, i.e.
\[
F^p:=F^p\Omega^n_{\E} = \pi^*\Omega^p_{\cT}\wedge \Omega^{n-p}_{\E}.
\]
This filtration gives a short exact sequence of complexes of sheaves:
\begin{equation}\label{exact_seq_of_complexes_induced_by_filtration}0 \to \pi^*\Omega^1_{\cT}\otimes_{\Oo_{\cT}} \Omega^{\bullet-1}_{\E/\cT}
\to \Omega^\bullet_{\E}/F^2
\to \Omega^\bullet_{\E/\cT}
\to 0.\end{equation}

The connecting homomorphism in the long exact sequence for the derived functor $\Rrr^n\!\pi_\ast$ associated with the short exact sequence \eqref{exact_seq_of_complexes_induced_by_filtration} has the form
\begin{equation}\label{nabla_introduction}\delta:\; \Rrr^n\pi_*\Omega_{\E/\cT}^\bullet
\longrightarrow
\Rrr^{n+1}\pi_*\bigl(\pi^*\Omega_{\cT}^1\otimes \Omega_{\E/\cT}^{\bullet-1}\bigr).\end{equation}
By the projection formula, we obtain an isomorphism
$$\Rrr^{n+1}\pi_*\bigl(\pi^*\Omega_{\cT}^1\otimes \Omega_{\E/\cT}^{\bullet-1}\bigr)
\;\cong\;
\Omega_{\cT}^1\otimes \Rrr^n\pi_*\Omega_{\E/\cT}^\bullet.$$
Taking this isomorphism into account, the map \eqref{nabla_introduction} takes the form
\begin{equation}\label{GM_connection}
\nabla:\; \cH^n:=\Rrr^n\pi_*\Omega_{\E/\cT}^\bullet
\longrightarrow
\Omega_{\cT}^1\otimes \cH^n
\end{equation}
and is called the {\it Gauss--Manin connection}.

Let us describe the map \eqref{GM_connection} explicitly. Let $s$ be a local section of the bundle $\cH^n$ over a trivializing open set $U\subset\cT$. Then, by definition, $s$ is represented by a relative closed $n$-form
\[
\eta \in \Gamma(\pi^{-1}(U),\Omega^n_{\mathcal E/\mathcal T}).
\]
Choose a lift
\[
\widetilde\eta \in \Gamma(\pi^{-1}(U),\Omega^n_{\E})
\]
of the form $\eta$ with respect to the projection $\Omega^n_{\E}\to \Omega^n_{\E/\cT}$. Compute the absolute differential
\[
\dd\widetilde\eta\in \Gamma(\pi^{-1}(U),\Omega^{n+1}_{\mathcal E}).
\]
Modulo terms containing at least two differentials from the base, the form $\dd\widetilde\eta$ defines a local section of the sheaf $\pi^*\Omega^1_{\cT}\otimes \Omega^n_{\E/\cT}$, i.e.
\[
\dd\widetilde\eta \equiv \sum_i \pi^*(\alpha_i)\wedge \beta_i \pmod{F^2}.
\]
Then
$$\nabla s =\sum_i \alpha_i\otimes [\beta_i]\in \Omega_{\cT}^1\otimes \cH^n.$$
Thus, one may think of the connection $\nabla$ as the connecting differential in the de Rham complex.

The Leibniz rule
$$\nabla(f\cdot s)=\dd f\otimes s+f\nabla(s)$$
follows from the Leibniz rule for the differential in the de Rham complex.

In fact, the connection $\nabla$ is the first differential in the spectral sequence of the filtered de Rham complex $\Omega^{\bullet}_{\E}$. Consequently, $\nabla^2=0$, and the Gauss--Manin connection is flat.

\subsection{Ramanujan Vector Field}

Let us now recall the geometric interpretation of the Ramanujan relations. Consider the family $\pi:\E \to \cT$ of elliptic curves
\begin{equation}
\label{eq:universal}
\E:\qquad
zy^2=4(x-t_1z)^3-t_2z^2(x-t_1z)-t_3z^3
\qquad\subset \mathbb{P}^2\times \cT,
\end{equation}
parameterized by the affine variety
\[
\cT \;=\; \Spec \Cc[t_1,t_2,t_3,\Delta^{-1}],
\]
where $\Delta \;=\; 27t_3^2-t_2^3$ is the discriminant of the polynomial on the right-hand side of the equation of the curve\footnote{In other words, the ring $\Cc[t_1,t_2,t_3,\Delta^{-1}]$ is the coordinate ring of the complement of the discriminant locus in $\Cc^3$.}.

Fix the basis $(\alpha, \omega)$ of the space $\Hhh^1_{\dR}(\E /\cT)$:
\begin{equation}
\alpha=\Bigl[\frac{\dd x}{y}\Bigr],\qquad
\omega=\Bigl[\frac{x\,\dd x}{y}\Bigr].
\end{equation}

The following statement is obtained by specializing the construction described in Section \ref{recollection_on_Gauss-Manin_connection} to hypersurfaces, via the Griffiths--Dwork reduction.

\begin{prop}[{\cite[Proposition 6]{Movasati2012}}]
Consider the $1$-form
\[
\theta \;:=\; 3t_3\,\dd t_2-2t_2\,\dd t_3 \in \Omega^1_{\cT}.
\]
Then, in the basis $(\alpha, \omega)$, the Gauss--Manin connection is given by the formula
\begin{equation}
\label{eq:GMmatrix}
\nabla
\begin{pmatrix}\alpha\\ \omega\end{pmatrix}
\;=\;
A\,
\begin{pmatrix}\alpha\\ \omega\end{pmatrix},
\end{equation}
where
\begin{equation}
\label{eq:A}
A \;=\; \frac{1}{\Delta}
\begin{pmatrix}
-\frac{3}{2}t_1\theta-\frac{1}{12}\dd\Delta & \frac{3}{2}\theta\\[4pt]
\Delta\,\dd t_1-\frac{1}{6}t_1\dd\Delta-\Bigl(\frac{3}{2}t_1^2+\frac{1}{8}t_2\Bigr)\theta &
\frac{3}{2}t_1\theta+\frac{1}{12}\dd\Delta
\end{pmatrix}
\end{equation}
and $\Delta = 27t_3^2-t_2^3$.
\end{prop}

\begin{defi}
The {\it Ramanujan vector field} on the base $\cT$ of the family $\pi:\E\to\cT$ of elliptic curves is the vector field $\vF$ satisfying the following horizontality conditions:
\begin{equation}
\label{horiz}
\nabla_{\vF}\alpha=-\omega,\qquad \nabla_{\vF}\omega=0.
\end{equation}
\end{defi}

\begin{prop}[{\cite[Proposition 7]{Movasati2012}}]
\label{thm:vF}
The Ramanujan vector field
$$\vF = F_1\frac{\partial}{\partial t_1}+F_2\frac{\partial}{\partial t_2}+F_3\frac{\partial}{\partial t_3}$$
on the base $\cT$ of the family $\pi:\E\to\cT$ of elliptic curves exists and is unique. The following formulas hold:
\begin{equation}
\label{eq:vFcomponents}
F_1=t_1^2-\frac{1}{12}t_2,\qquad
F_2=4t_1t_2-6t_3,\qquad
F_3=6t_1t_3-\frac{1}{3}t_2^2.
\end{equation}
\end{prop}

For completeness, let us give a proof of this proposition (it was omitted in \cite{Movasati2012}).

\begin{proof}
Let
\[
v=F_1\partial_{t_1}+F_2\partial_{t_2}+F_3\partial_{t_3}
\]
be a vector field on the base $\cT$ satisfying the horizontality conditions \eqref{horiz}. Evaluating \eqref{eq:GMmatrix} on the vector field $v$, we obtain
\[
\nabla_v
\begin{pmatrix}\alpha\\ \omega\end{pmatrix}
\;=\;
A(v)\,
\begin{pmatrix}\alpha\\ \omega\end{pmatrix},
\qquad
A(v)\in M_2(\mathcal \Oo_{\cT}).
\]

On the one hand, the conditions $\nabla_v\alpha=-\omega$ and $\nabla_v\omega=0$ are equivalent to the fact that
\begin{equation}
\label{eq:Av}
A(v)=
\begin{pmatrix}
0 & -1\\
0 & 0
\end{pmatrix}.
\end{equation}
On the other hand, the matrix $A(v)$ is given by formula \eqref{eq:A}.

Comparing the upper-right entries of the matrices \eqref{eq:A} and \eqref{eq:Av}, we obtain the equation
\begin{equation}
\label{eq:theta1}
\frac{3}{2}\frac{\theta(v)}{\Delta}=-1
\qquad\Longrightarrow\qquad
\theta(v)=-\frac{2}{3}\Delta.
\end{equation}
The upper-left entry of the matrix $A(v)$ gives
\begin{equation}\label{eq:theta2}
\frac{-\frac{3}{2}t_1\theta(v)-\frac{1}{12}\dd\Delta(v)}{\Delta}=0
\qquad\Longrightarrow\qquad
\dd\Delta(v)= -18t_1\,\theta(v)=12t_1\Delta.
\end{equation}
Now let us compute $\theta(v)$ and $\dd\Delta(v)$ in terms of $F_2$ and $F_3$:
\[
\theta(v)=3t_3F_2-2t_2F_3,\qquad
\dd\Delta(v)= (54t_3)F_3-(3t_2^2)F_2.
\]

From equations \eqref{eq:theta1} and \eqref{eq:theta2}, we obtain the following linear system of equations for $F_2$ and $F_3$:
\begin{equation}\label{linear_system_on_F_2_and_F_3}
\begin{cases}
3t_3F_2-2t_2F_3&=-\frac{2}{3}\Delta,\\
-3t_2^2F_2+54t_3F_3&=12t_1\Delta.
\end{cases}
\end{equation}
Solving system \eqref{linear_system_on_F_2_and_F_3} gives
\[
F_2=4t_1t_2-6t_3
\qquad\text{and}\qquad
F_3=6t_1t_3-\frac13t_2^2.
\]

Finally, for the lower-left entry of the matrix \eqref{eq:Av}, we have
\begin{equation}\label{F_1_equation}
\frac{\Delta\,\dd t_1(v)-\frac{1}{6}t_1\dd\Delta(v)-\bigl(\frac{3}{2}t_1^2+\frac{1}{8}t_2\bigr)\theta(v)}{\Delta}=0.
\end{equation}
Using the facts that $\dd t_1(v)=F_1$, $\dd\Delta(v)=12t_1\Delta$, and $\theta(v)=-\frac23\Delta$, equation \eqref{F_1_equation} yields
\[
F_1=t_1^2-\frac{1}{12}t_2.
\]
We have thus obtained formulas \eqref{eq:vFcomponents} for the Ramanujan vector field $\vF$. Uniqueness follows automatically.
\end{proof}

Define the curve
\begin{equation}
\label{eq:Phi}
\Phi(\tau)\;:=\;\Bigl(\frac{1}{12}E_2(\tau),\ \frac{1}{12}E_4(\tau),\ \frac{1}{216}E_6(\tau)\Bigr).
\end{equation}
on the base $\cT$.

\begin{prop}
\label{prop:Ramanujan}
The Ramanujan system of differential equations \eqref{Ramanujan} for the Eisenstein series is equivalent to the condition that $\Phi$ is an integral curve of the vector field $\vF$, i.e.
\[
\Phi'(\tau)=\vF_{\Phi(\tau)},
\]
where $(-)' = 1/(2\pi i)\dd/\dd\tau$.
\end{prop}

\begin{proof}
Denote the components of $\Phi(\tau)$ by $t_1(\tau)$, $t_2(\tau)$, and $t_3(\tau)$. Using \eqref{Ramanujan} and \eqref{eq:Phi}, we compute
\[
t_1'=\frac{1}{12}E_2'=\frac{1}{144}(E_2^2-E_4)=t_1^2-\frac{1}{12}t_2,
\]
and similarly
\[
t_2'=4t_1t_2-6t_3,
\qquad
t_3'=6t_1t_3-\frac{1}{3}t_2^2.
\]
These are exactly the coordinate expressions \eqref{eq:vFcomponents}, hence
\[
\Phi'=\vF_{\Phi(\tau)}.
\]

Conversely, if $\Phi$ is an integral curve of $\vF$, then the three equalities
\[
t_i'=F_i(t_1,t_2,t_3)
\]
together with \eqref{eq:Phi} are equivalent to \eqref{Ramanujan} after the substitutions
\[
E_2 :=12t_1,\qquad E_4 :=12t_2,\qquad E_6 :=216t_3.
\]
\end{proof}

Note that, in addition to the vector field $\vF$ on the base
\[
\cT = \Cc^3\backslash\{\Delta=0\}
\qquad
(\text{recall that } \Delta = 27t_3^2-t_2^3)
\]
of the family $\pi:\E\to\cT$ of elliptic curves, there are also two vector fields
$$\uF = -6t_3 \frac{\partial}{\partial t_3}
    - 4t_2 \frac{\partial}{\partial t_2}
    - 2t_1 \frac{\partial}{\partial t_1},
\qquad
\wF = \frac{\partial}{\partial t_1},$$
such that the vector fields $\uF$, $\vF$, and $\wF$ are linearly independent at every point of the base $\cT$. It is easy to check that each of these vector fields is tangent to the discriminant locus $\{\Delta = 0\}$. The theory of singularities arising in the context of holomorphic vector fields tangent to discriminants is discussed in detail in the monograph \cite[Section 4.1]{Arnold90}.

\section{Frobenius Structures}\label{Frobenius_Structures}

The theory of Frobenius algebras intersects the theory of $n$-valued groups. In this section, concrete examples of these intersections will be considered. It will be shown that the associativity condition $4k_8=k_4^2-k_6k_2$ for the two-valued multiplication can be interpreted as the associativity condition of a certain Frobenius structure on $\R^3$.

\subsection{Frobenius $n$-Homomorphisms and Representation Theory}

In \cite[Section 2]{Buchstaber_Rees02}, for every unital associative algebra over $\Cc$ and every positive integer $n$, the notion of a Frobenius $n$-homomorphism $\Phi_n$ was introduced. Consider the linear map assigning to a tracial linear functional $f:A\to \Cc$ (i.e. $f(a_1a_2)=f(a_2a_1)$ for all $a_1,a_2\in A$) the linear map
$$\Phi_n(f):A^{\otimes n}\to \Cc$$
by the following recursive rule:
\begin{equation}
\begin{aligned}
\Phi_1(f) &= f,\\
\Phi_2(f)(a_1\otimes a_2) &= f(a_1)f(a_2) - f(a_1a_2),\\
\Phi_{n+1}(f)(a_1\otimes\cdots\otimes a_{n+1})
&=
f(a_1)\,\Phi_n(f)(a_2\otimes\cdots\otimes a_{n+1})\\
&\quad-\sum_{i=2}^{n+1}
\Phi_n(f)(a_2\otimes\cdots\otimes a_{i-1}\otimes (a_1a_i)\otimes a_{i+1}\otimes\cdots\otimes a_{n+1}).
\end{aligned}
\end{equation}

\begin{defi}\label{Frob_n-hom}
A linear functional $f:A\to \Cc$ is called a {\it Frobenius $n$-homomorphism} if $f(1)=n$ and $\Phi_{n+1}(f)=0$.
\end{defi}

Definition \ref{Frob_n-hom} is based on Frobenius's results in the representation theory of finite groups. To make this connection more transparent, let us rewrite the definition of the map $\Phi_n$ in an equivalent form (see \cite[Section 3, page 134]{Buchstaber_Rees04}):
\begin{equation}\label{Phi_n_as_sum_over_cycles}
\Phi_n(f) = \sum\limits_{\sigma\in S_n}\sgn(\sigma)f_{\cyc(\sigma)},
\end{equation}
where $\sgn$ is the sign of a permutation,
$$f_{\cyc(\sigma)}(a_1,\dots,a_n) = f_{\gamma_1}(a_1,\dots,a_n)\cdots f_{\gamma_r}(a_1,\dots,a_n),$$
$\sigma = \gamma_1\cdots\gamma_r$ is the decomposition of the permutation into disjoint cycles, and for each cycle $\gamma=(i_1,\dots,i_m)$ one sets
$$f_{\gamma}=f(a_{i_1}\cdots a_{i_m}).$$

If one takes as the algebra $A$ the matrix algebra $\Mat_d(R)$ over an arbitrary commutative $\Q$-algebra, and as the map $f$ the trace $\tr$, then one obtains formula (4) from the survey \cite{Conrad}:
\begin{equation}\label{fancy_determinant_formula}
\det(M) = \frac{1}{d!}\sum\sgn(\sigma)\tr_{\sigma}(M)
\end{equation}
for every matrix $M$. As Conrad notes, a similar formula was used by Frobenius in \cite[Equation 8, Section 3]{Frobenius1896}.

In \cite[page 134]{Buchstaber_Rees04}, the polynomials
\begin{equation}\label{F_n_polynomials}
F_n(s_1, ..., s_n) = \sum\limits_{\sigma\in S_n}\prod\limits_{k=1}^n\left((-1)^{k+1}s_k\right)^{m_k(\sigma)},
\end{equation}
were introduced, where $m_k(\sigma)$ denotes the number of cycles of length $k$ in the decomposition of the permutation $\sigma$ into disjoint cycles. It is noted that when $s_k=f(a^k)$, the polynomial $F_n(s_1,\dots,s_n)$ becomes the polarization of the $n$-linear map $\Phi_n(f)$, i.e. the equality
\begin{equation}\label{polarization_formula}
\Phi_n(f)(a,\dots,a)=F_n(f(a), ..., f(a^n))
\end{equation}
holds.

From \eqref{fancy_determinant_formula}, \eqref{F_n_polynomials}, and \eqref{polarization_formula}, for every matrix $M\in\Mat_n(R)$ one obtains the equalities
\begin{equation}
\det(M) =  \frac{1}{n!} F_n(\tr(M),\tr(M^2),\dots,\tr(M^n)) = \frac{1}{n!}\Phi_n(\tr)(M,\dots,M).
\end{equation}

\begin{defi}
Let $M_n$ be the $\Z$-module generated by the characters of irreducible representations of the symmetric group $S_n$ over $\Cc$. Let $\chi\in M_n$. Then the {\it Frobenius characteristic map} is defined by
\begin{equation}\label{Frobenius_characteristic_map}
\chr(\chi)=\frac{1}{n!}\sum\limits_{\sigma\in S_n}\chi(\sigma)p_{\cyc(\sigma)},
\end{equation}
where $p_{\cyc(\sigma)} = p_1^{m_1} p_2^{m_2} \cdots p_n^{m_n}$, for each $k$ the expression
$$p_k = \sum_i x_i^k\in\Cc[x_1,\dots,x_n]$$
is the Newton power-sum polynomial, and $m_k = m_k(\sigma)$ is the number of cycles of length $k$ in the decomposition of the permutation $\sigma$ into disjoint cycles.
\end{defi}

It is easy to see that formula \eqref{Frobenius_characteristic_map} is equivalent to the formula
$$\chr(\chi)=\sum\limits_{\lambda\vdash n}z_{\lambda}^{-1}\chi(\lambda)p_{\lambda},$$
where $z_{\lambda}=n!/|K_{\lambda}|$, $K_\lambda$ is the conjugacy class of permutations of cycle type $\lambda = (\lambda_1,\dots,\lambda_s)\vdash n$, and $p_\lambda=p_1^{\lambda_1}\cdots p_s^{\lambda_s}$.

\begin{example}
The map $\Phi_n$ defined by formula \eqref{Phi_n_as_sum_over_cycles} is a specialization of the Frobenius characteristic map. Namely, the specialization $p_k:=f(a^k)$ (for each $k=1,\dots,n$) and $\chi = \sgn$ in formula \eqref{Frobenius_characteristic_map} yields formula \eqref{Phi_n_as_sum_over_cycles} (modulo multiplication by $n!$).
\end{example}

\begin{example}\label{higher_n-characters}
Let $A=\Cc\!G$ be the group algebra of a finite group, and let $\chi$ be a character of a complex representation of $G$. Then for each positive integer $n$, the expression $\Phi_n(\chi)$ was called by Frobenius the {\it higher $n$-character} \cite{Frobenius1896, Vazirani03}.
\end{example}

A remarkable property of the Frobenius characteristic map is that it is a ring homomorphism and establishes an isometric isomorphism between the ring of characters of the symmetric group $S_n$ (with the standard inner product of characters) and the ring of homogeneous symmetric functions of degree $n$ (with the Hall inner product, given by the formula $\langle s_\lambda, s_\mu \rangle = \delta_{\lambda\mu}$) \cite[Chapter IV, Section 4]{Macdonald95}.

Let $\rho:G\to\GL_n(\Cc)$ be a representation of a finite group $G$, $\chi_\rho$ its character, $\{X_g\}_{g\in G}$ commuting formal variables indexed by the elements of the group $G$, $a=\sum\limits_{g\in G}X_g g\in\Cc\!G[\{X_g\}]$, and, finally, $M=M_\rho = \rho(a)$. Under this specialization, we obtain the formula
\begin{equation}\label{det_rho}
\det(\rho(a)) = \frac{1}{n!}\Phi_n(\chi_{\rho})(a, ..., a).
\end{equation}

\begin{defi}
The {\it group determinant $\Theta(G)$} of a finite group $G$ is the expression (in the notation above)
$$\Theta(G) = \det(M_{\mathrm{\rho_{\reg}}}),$$
where $\rho_{\reg}$ is the regular representation of the group $G$.
\end{defi}

\begin{example}
Taking in formula \eqref{det_rho} the regular representation $\rho_{\reg}$ of the group $G$ as $\rho$, we obtain the equality
$$\Theta(G)=\frac{1}{|G|!}\,\Phi_{|G|}(\chi_{\mathrm{reg}})(a,\dots,a).$$
\end{example}

Let us recall Frobenius's classical result on the group determinant:

\begin{theorem}[{(Frobenius, \cite{Frobenius1896})}]
\begin{enumerate}[\bf (i)]
\item For any finite group $G$, the following equality holds:
\begin{equation}\label{group_determinant_formula}
\Theta(G)=\prod_{\rho\ \mathrm{irred}}
\det\Bigl(\sum_{g\in G}X_g\rho(g)\Bigr)^{\deg\rho},
\end{equation}
where the product is taken over all irreducible complex representations $\rho$ of the group $G$.

\item For each irreducible representation $\rho$ of the group $G$, the polynomial
$$P_\rho = \det\Bigl(\sum\limits_{g\in G}X_g\rho(g)\Bigr)$$
is irreducible.

\item There is a bijection between the set $\{\rho\}$ of irreducible representations and the set $\{P_\rho = \det(\sum\limits_{g\in G}X_g\rho(g))\}$ of the corresponding polynomials.
\end{enumerate}
\end{theorem}

From formulas \eqref{det_rho} and \eqref{group_determinant_formula}, we obtain the following result on the relation between the maps $\Phi_n$ and the group determinant.

\begin{prop}
For any finite group $G$, the following formula holds:
$$\Theta(G)
=\prod_{\rho\ \mathrm{irred}}
\left(\frac{1}{(\deg\rho)!}\,\Phi_{\deg\rho}(\chi_\rho)(a,\dots,a)\right)^{\deg\rho},$$
where $a=\sum\limits_{g\in G}X_g g\in\Cc\!G[\{X_g\}]$.
\end{prop}

Frobenius $n$-homomorphisms may be thought of as characters of $n$-dimensional representations of finite groups. As is well known, every $n$-dimensional complex representation of a finite abelian group is a sum of characters of one-dimensional representations. Similarly, every Frobenius $n$-homomorphism is a sum of $n$ ordinary homomorphisms. This observation, in algebraic-geometric and topological contexts, was obtained in the works of Buchstaber and Rees \cite{Buchstaber_Rees02, Buchstaber_Rees04}. A general algebraic formulation is due to Gugnin \cite[Theorem 1]{Gugnin07}.

The identities for the Frobenius characteristic map discussed above arise from certain specializations of Schur functors. An abstract categorical context for Schur functors may be found in \cite{Baez24}.

\subsection{Frobenius Algebras and $n$-Homomorphisms}

Let us briefly recall the main definitions and constructions from the theory of Frobenius algebras in order to fix notation.

\begin{defi}
A {\it Frobenius algebra $A$} (or a {\it Frobenius pair $(A,\theta)$}) is a $\kk$-algebra $A$ with multiplication $\mu$ and unit $e$, equipped with a linear map $\theta: A\to \kk$, which we call the Frobenius functional, such that the bilinear form
\[
\eta=\langle-,-\rangle:=\theta\circ\mu(-,-)
\]
is nondegenerate, i.e. the morphism
$$u\mapsto ( v\mapsto \langle v, u\rangle)$$
is an isomorphism $V\to V^\ast$, where $V$ is the underlying vector space of the algebra $A$.
\end{defi}

It follows from the definition that every Frobenius $\kk$-algebra is finite-dimensional. Indeed, if $A$ were an infinite-dimensional vector space, then by the Erd\H{o}s--Kaplansky theorem
\[
\dim A^\ast=|\kk^A|=|\kk|^{\dim A}
\]
(see, for example, \cite[Chapter IX, Section 5, Theorem 2]{Jacobson53}). By Cantor's theorem,
\[
|\kk|^{\dim A}>\dim A.
\]
Hence $\dim A^\ast>\dim A$, contradicting the isomorphism $A\cong A^\ast$.

It is easy to see that the form $\langle-,-\rangle$ and the multiplication $\mu$ satisfy the following Frobenius condition for all $a,b,c\in A$:
\[
 \langle ab,c\rangle=\langle a, bc\rangle.
\]

Associated with every Frobenius algebra $(A,\theta)$ are the linear form $\theta$, the bilinear form $\eta$, and the trilinear form $\xi$. The latter is defined by
\[
\xi(a\otimes b\otimes c)=\theta(abc).
\]

\begin{example}\label{examples_of_Frobenius_algebras}
The group algebra $\kk\!G$ of a finite group $G$, with functional
\[
\theta\Bigl(\sum\limits_{g\in G}\alpha_g g\Bigr) = \alpha_e,
\]
where $e$ is the identity element of $G$, is a Frobenius algebra. Another example is given by the cohomology algebra of a smooth closed oriented manifold $X$ of dimension $2n$ with the cup product. In this case,
\[
\theta\Bigl(\sum\limits_i a_i\Bigr)=\int\limits_X a_n.
\]
A further example is the Frobenius algebra on the affine space $\kk^n$ with the standard coordinatewise product and Frobenius functional
\[
\theta(a_1,\dots,a_n)=\sum\limits_{i=1}^n\lambda_i a_i
\]
for a fixed collection of nonzero elements $\lambda_i\in\kk$.
\end{example}

Every finite-dimensional Hopf algebra admits a Frobenius algebra structure \cite[Chapter V]{Sweedler69}.

Let us record the following statement, well known in the theory of two-dimensional TQFT, for future reference. In the proof, unlike the ``picture method'' from \cite[Proposition 2.3.22]{Kock03}, we will follow the standard algebraic approach, since explicit formulas will be important for us. The proof follows the unpublished lectures of Rolf Farnsteiner \cite{Farnsteiner}.

\begin{prop}
\label{coalgebra_strucure_in_Frobenius_algebra} 
\begin{enumerate}[\bf (i)]
\item Every Frobenius $\kk$-algebra $(A,\theta)$ of dimension $n$ carries a unique coalgebra structure $(A,\Delta,\theta)$ with counit $\theta$. Moreover, if
\[
\Delta(1)=\sum\limits_{i=1}^m a_i\otimes b_i
\]
for some $m$, then
\begin{equation}\label{comult}\Delta(x)=\sum\limits_{i=1}^m a_ix\otimes b_i=\sum\limits_{i=1}^m a_i\otimes xb_i.\end{equation}

\item Let $(x_1,\dots,x_n)$ and $(y_1,\dots,y_n)$ be two dual bases, i.e.
\[
\langle x_i, y_j \rangle = \delta_{ij}.
\]
Then
\[
\Delta(x)=\sum\limits_{i=1}^n x_ix\otimes y_i=\sum\limits_{i=1}^n x_i\otimes xy_i.
\]
\end{enumerate}
\end{prop}

\begin{proof}
Let us prove part {\bf (i)}. Let $A$ be a $\kk$-algebra. We will show that $A$ is a Frobenius algebra if and only if there exists a coalgebra structure $(A,\Delta,\theta)$ with comultiplication \eqref{comult}.

Suppose first that $(A,\Delta,\theta)$ is a coalgebra with comultiplication \eqref{comult}. Introduce the bilinear form
\[
\langle x,y\rangle=\theta(xy)
\qquad\text{for all }x,y\in A.
\]
Let us check that this form is nondegenerate. Suppose that $x\in A$ is such that
\[
\langle a,x\rangle=0
\qquad\text{for all }a\in A.
\]
Then (here $\overline{\otimes}$ denotes the composition of the tensor product with the subsequent multiplication in the field $\kk$)
\[
x = (\theta\overline{\otimes}\id_{A})\circ\Delta(x)=\sum\limits_{i=1}^m\theta(a_ix)\, b_i=\sum\limits_{i=1}^m\langle a_i, x \rangle b_i=0.
\]
Hence the form $\langle-,-\rangle$ is nondegenerate, and therefore $(A,\theta)$ is a Frobenius algebra.

Now let $(A,\theta)$ be a Frobenius algebra. On the dual space $A^\ast$ there is a natural coalgebra structure
\[
\mu^\ast:A^\ast\to A^\ast\otimes_{\kk} A^\ast
\]
with counit
\[
e^\ast:A^\ast\to \kk,
\]
defined by the rules
\[
\begin{aligned}
\mu^\ast(f)=\sum\limits_{i=1}^m f_i\otimes g_i
&\Leftrightarrow
f(ab)=\sum\limits_{i=1}^m f_i(a)g_i(b)\quad \forall a,b\in A,\\
e^{\ast}(f)&=f(e)\quad \forall f\in A^\ast.
\end{aligned}
\]
Let
\[
\phi:A\to A^\ast,\qquad a\mapsto \langle a,-\rangle
\]
be the isomorphism induced by the Frobenius form. Define the required comultiplication by
\[
\Delta := (\phi^{-1}\otimes \phi^{-1})\circ \mu^{\ast}\circ \phi.
\]
Equivalently, the comultiplication map $\Delta$ is determined by the condition
\begin{equation}\label{diagonal_definition}
\Delta(x)=\sum\limits_{i=1}^m\alpha_i\otimes \beta_i
\;\Longleftrightarrow\;
\langle x, yz \rangle = \sum\limits_{i=1}^m \langle \alpha_i, y \rangle\langle \beta_i, z \rangle
\quad \forall y,z\in A.
\end{equation}

Let us prove formulas \eqref{comult}. Suppose
\begin{equation}\label{Delta_of_1}\Delta(1)=\sum\limits_{i=1}^m a_i\otimes b_i
\end{equation}
for some $a_i,b_i\in A$, $i=1,\dots,m$. Let $\nu$ be the Nakayama automorphism of the Frobenius algebra $A$, i.e. the automorphism
\[
\nu: A\to A
\]
such that
\[
\langle y, x \rangle = \langle \nu(x), y \rangle
\qquad\text{for all }x,y\in A.
\]
Then
\begin{equation}\label{diagonal_formula_proof}
\begin{aligned}
\sum_{i=1}^{m} \langle a_i,y\rangle\langle x b_i, z\rangle
=
\sum_{i=1}^{m} &\langle a_i,y\rangle\langle x, b_i z\rangle
=
\sum_{i=1}^{m} \langle a_i,y\rangle\langle b_i z, \nu^{-1}(x)\rangle
=
\sum_{i=1}^{m} \langle a_i,y\rangle\langle b_i, z \nu^{-1}(x)\rangle\\
&=
\langle 1, y z \nu^{-1}(x)\rangle
=
\langle y z, \nu^{-1}(x)\rangle
=
\langle x, y z\rangle
\qquad \forall x,y,z\in A.
\end{aligned}
\end{equation}
From \eqref{diagonal_formula_proof} and \eqref{diagonal_definition}, we obtain
\[
\Delta(x)=\sum\limits_{i=1}^m a_i\otimes xb_i.
\]
Similarly, one shows that
\[
\Delta(x)=\sum\limits_{i=1}^m a_ix\otimes b_i.
\]

Now let us prove part {\bf (ii)}. Suppose
\[
\Delta(1)=\sum\limits_{i=1}^m a_i\otimes b_i,
\qquad m\leq n.
\]
We have
\begin{equation}\label{x_decomposition}
x=(\id_A\overline{\otimes} \theta)\circ\Delta(x)=\sum\limits_{i=1}^m\langle x,b_i \rangle a_i\quad \forall x\in A.
\end{equation}
Hence the set $\{a_1,\dots,a_m\}$ spans the vector space $A$, so it is a basis of the algebra $A$. It follows from \eqref{x_decomposition} that $(a_1,\dots,a_n)$ and $(b_1,\dots,b_n)$ are dual bases. The value of the expression
\[
\Delta(1)=\sum\limits_{i=1}^n a_i\otimes b_i
\]
does not depend on the choice of the pair of dual bases $(a_1,\dots,a_n)$ and $(b_1,\dots,b_n)$, since $\Delta(1)$ is the preimage of the element $\id_A$ under the isomorphism
\[
\begin{aligned}
\psi \colon A \otimes_{\Bbbk} A &\longrightarrow \operatorname{Hom}_{\Bbbk}(A,A),\\
a \otimes b &\longmapsto \bigl( x \mapsto \langle a, x\rangle\, b \bigr).
\end{aligned}
\]
\end{proof}

Every commutative Frobenius algebra encodes a two-dimensional topological quantum field theory, yielding an equivalence between the category of two-dimensional topological quantum field theories and the category of commutative Frobenius algebras. Moreover, this equivalence makes it possible to obtain transparent topological proofs of certain identities in Frobenius algebras \cite{Kock03}.

Let $(A,\theta)$ be a commutative Frobenius algebra with basis $\{e_i\mid i\in I\}$. As was shown in \cite[page 138]{Buchstaber_Rees04}, the values $\theta(e_i)$, $\theta(e_ie_j)$, and $\theta(e_ie_je_k)$ can be expressed in terms of $\Phi_s(e_\ell):=\Phi_s(\theta)(e_\ell)$, $s=1,2,3$; namely, for all $i,j,k\in I$:
\[
\begin{aligned}
\theta(e_i)&=\Phi_1(e_i),\\
\theta(e_ie_j)&=\Phi_1(e_i)\Phi_1(e_j)-\Phi_2(e_i,e_j),\\ 
2\theta(e_ie_je_k)
&=
\Phi_3(e_i,e_j,e_k)
+2\Phi_1(e_i)\Phi_1(e_j)\Phi_1(e_k)
-\Phi_1(e_i)\,\Phi_2(e_j,e_k)\\
&-\Phi_1(e_j)\,\Phi_2(e_i,e_k)
-\Phi_1(e_k)\,\Phi_2(e_i,e_j).
\end{aligned}
\]

Recall that the structure constants $a_{ij}^k$ of the algebra $A$ are defined by the equalities
\[
e_ie_j = \sum\limits_{k}a_{ij}^ke_k
\qquad\text{for all }i,j\in I.
\]
The result \cite[Proposition 19]{Buchstaber_Rees04} states that the structure constants $a_{ij}^k$ of every commutative Frobenius algebra $A$ are recovered from the values of the $k$-characters (in the terminology of Frobenius, see Example \ref{higher_n-characters}) $\Phi_k(\theta)$ on the basis elements $\{e_i\mid i\in I\}$.

\subsection{Dubrovin--Frobenius Structures and 2-Valued Groups}\label{Dubrovin_Frobenius_structures_and_2-valued_groups}

In the case where a smooth manifold carries, in each tangent space, a Frobenius algebra structure depending smoothly on the point, one speaks of a Dubrovin--Frobenius structure on this manifold. This notion was first introduced by Dubrovin \cite{Dubrovin96}. Let us recall the main definitions and constructions from the theory of Dubrovin--Frobenius structures on manifolds, following \cite{Natanzon25} (see also \cite{Dubrovin96}).

\begin{defi}
Let $M$ be a smooth real or complex manifold. We say that a {\it Dubrovin--Frobenius structure} is given on $M$ if for each point $p\in M$ the tangent space $M_p=T_pM$ is endowed with the structure of a semisimple Frobenius algebra $(M_p,\theta_p)$ (i.e. this algebra is isomorphic to an algebra structure on the affine space $\kk^n$, see Example \ref{examples_of_Frobenius_algebras}) and the following conditions hold:
\begin{enumerate}[\bf (i)]
\item The tensors $\theta=\{\theta_p\mid p\in M\}$, $\eta(a,b)=\theta(ab)$, and $\xi(a,b,c)=\theta(abc)$ are smooth and $d\theta = 0$.
\item The pseudometric $\eta$ is flat (i.e. its Gram matrix is locally constant in suitable systems of local coordinates); the unit field $e$ is parallel (i.e. $\nabla e=0$, where $\nabla$ is the Levi--Civita connection for $\eta$).
\item There exists a covering $M=\bigcup\limits_{\alpha}U_{\alpha}$ by local charts
\[
(x_{\alpha}^1,\dots,x_{\alpha}^n):U_{\alpha}\to \kk^n,
\]
such that the collection
\[
\left( \frac{\partial}{\partial x_{\alpha}^1},\dots,\frac{\partial}{\partial x_{\alpha}^n} \right)
\]
is the canonical basis of the tangent space $M_p$ at every point $p\in U_\alpha$ (i.e.
\[
\frac{\partial}{\partial x_\alpha^i}\cdot\frac{\partial}{\partial x_\alpha^j} = \delta_{ij}.
\]
\item The Euler field
\[
E=\sum\limits_{i=1}^n x^i_\alpha\frac{\partial}{\partial x_\alpha^i}
\]
does not depend on $\alpha$, and moreover $\nabla\nabla E=0$ and
\[
L_E\theta =(m+1)\theta,
\]
where $L_E$ is the Lie derivative along the field $E$.
\end{enumerate}
\end{defi}

Every Dubrovin--Frobenius structure on a manifold $M$ admits, near every point $x\in M$, a local coordinate system
\[
(t^1,\dots,t^n):U\to\kk^n
\]
such that for some function $F(t^1,\dots,t^n)$, called the {\it potential}, one has
\[
\theta\left(\frac{\partial}{\partial t^i}\cdot\frac{\partial}{\partial t^j}\cdot\frac{\partial}{\partial t^k}\right) = \frac{\partial^3 F}{\partial t^i\partial t^j\partial t^k}.
\]
Such a coordinate system is called {\it flat quasi-homogeneous}. In these coordinates, the pseudometric, the unit vector field, and the Euler vector field take the form
\[
\begin{aligned}
\eta &= \sum\limits_{\alpha,\beta}\delta_{\alpha+\beta, n+1}\,dt^{\alpha}\otimes dt^{\beta},\quad
e = \frac{\partial}{\partial t^{\epsilon}},\\
E&=\sum\limits_{\alpha=1}^n(d_{\alpha}t^{\alpha}+r_{\alpha})\frac{\partial}{\partial t^{\alpha}},
\end{aligned}
\]
where $1\leq\epsilon\leq n$ is some natural number (in Dubrovin's notation one chooses $\epsilon = 1$), $d_{\epsilon} = 1$, $d_\alpha$ and $r_\alpha$ are constants from the field $\kk$, and $d_\alpha r_\alpha = 0$. The numbers $d_1,\dots,d_n$ are called the {\it degrees of the manifold $M$}.

For a given Euler vector field
\[
E=\sum\limits_{\alpha}(d_{\alpha}t^{\alpha}+r_\alpha)\frac{\partial}{\partial t^{\alpha}}
\]
in flat quasi-homogeneous coordinates, under the condition
\[
d_\alpha+d_{n+1-\alpha}=m+2,
\]
and for a potential $F(t^1,\dots,t^n)$, one says that the pair $(F, E)$ is a {\it solution of the WDVV equations} (these equations first appeared in the field-theoretic works of Witten \cite{Witten90} and Dijkgraaf, Verlinde, Verlinde \cite{DVV91}) if the following conditions hold:
\begin{enumerate}[\bf (i)]
\item The associativity equation
\[
\sum_{\nu=1}^n
\frac{\partial^3 F}{\partial t^\alpha\,\partial t^\beta\,\partial t^\nu}\,
\frac{\partial^3 F}{\partial t^\gamma\,\partial t^\delta\,\partial t^{n+1-\nu}}
=
\sum_{\nu=1}^n
\frac{\partial^3 F}{\partial t^\gamma\,\partial t^\beta\,\partial t^\nu}\,
\frac{\partial^3 F}{\partial t^\alpha\,\partial t^\delta\,\partial t^{n+1-\nu}},
\qquad
\forall\,\alpha,\beta,\gamma,\delta.
\]

\item The normalization condition
\[
\frac{\partial^3 F}{\partial t^\epsilon\,\partial t^\alpha\,\partial t^\beta}
= \eta\left(\frac{\partial}{\partial t^{\alpha}},\frac{\partial}{\partial t^{\beta}}\right)=
\delta_{\alpha+\beta,n+1},
\qquad
\forall\,\alpha,\beta.
\]

\item The homogeneity conditions
\[
L_E F=(m+3)F+\sum_{\alpha,\beta} A_{\alpha\beta}\,t^\alpha t^\beta
+\sum_{\alpha} B_\alpha\,t^\alpha + C,
\]
where $A_{\alpha\beta}$, $B_\alpha$, and $C$ are constants.
\end{enumerate}

If $M$ is a manifold with a Dubrovin--Frobenius structure and flat quasi-homogeneous coordinates with potential $F$ and Euler vector field $E$ are chosen on a domain $U\subset M$, then the multiplication in the tangent space $M_p$ takes the form
\begin{equation}\label{product_of_t}
\frac{\partial}{\partial t^{\alpha}}\cdot\frac{\partial}{\partial t^{\beta}} = \sum\limits_{\gamma}\frac{\partial^3 F}{\partial t^{\alpha}\partial t^{\beta}\partial t^{n+1-\gamma}}\cdot\frac{\partial}{\partial t^{\gamma}},
\end{equation}
and the pair $(F, E)$ is a solution of the WDVV equations.

The converse statement is also true. Let $(F, E)$ be a solution of the WDVV equations on a domain $U\subset\kk^n$. Suppose that in each tangent space $T_pU$ the operation \eqref{product_of_t} is given and
\[
g\left(\frac{\partial}{\partial t^{\alpha}},\frac{\partial}{\partial t^{\beta}}\right)=\delta_{\alpha+\beta,n+1}.
\]
Then this operation endows $T_pU$ with the structure of a Frobenius algebra.

If $\theta(e)=0$ (as often happens), then one may, without loss of generality, take $\epsilon = 1$, $e=\frac{\partial}{\partial t^1}$. In this case, the normalization condition
\[
\frac{\partial^3F}{\partial t^1\partial t^\alpha\partial t^\beta} = \delta_{\alpha+\beta, n+1}
\]
implies that the potential has the form
\[
F(t^1,\dots,t^n) = \frac{1}{2}(t^1)^2t^n+\frac{1}{2}t^1\sum\limits_{\alpha=2}^{n-1}t^{\alpha}t^{n+1-\alpha} + f(t^2,\dots,t^n).
\]

Let us now restrict to the case of a domain $U$ in $\R^3$ with fixed basis
\[
e_k=\frac{\partial}{\partial t^k},\qquad k=1,2,3,
\]
unit field $e=e_1$, and potential
\begin{equation}\label{polynomial_potetnial_of_order_3}
F(t) = \frac{1}{2}(t^1)^2t^3+\frac{1}{2}t^1(t^2)^2 + f(t^2,t^3).
\end{equation}
The metric in these flat quasi-homogeneous coordinates has the form
\begin{equation}\label{metric_matrix}
\eta = \left (\begin{matrix}0 & 0& 1\\ 0& 1& 0\\ 1& 0& 0 \end{matrix} \right ).
\end{equation}

The multiplication table for the basis elements is
\begin{equation}\label{mult_table_2}
e_2^2=a e_1+b e_2+e_3,\qquad
e_2e_3=c e_1+a e_2,\qquad
e_3^2=d e_1+c e_2,
\end{equation}
where
\[
a:=f_{xxy},\qquad b:=f_{xxx},\qquad c:=f_{xyy},\qquad d:=f_{yyy}.
\]
Recall that $e_1$ is the identity field.

From the form \eqref{metric_matrix} of the metric
\[
\eta_{\alpha\beta}=c_{1\alpha\beta}=\frac{\partial^3 F}{\partial t^1\,\partial t^\alpha\,\partial t^\beta}
\]
it follows that $\theta(e_\alpha)=\eta_{\alpha 1}$ and
\[
\theta(e_1) = \theta(e_2) = 0,\qquad \theta(e_3)=1.
\]

For each point $p\in U$, the Frobenius algebra $T_pU$ is isomorphic to the algebra (where $\mu$ is a formal variable)
\[
\R[\mu]/(\mu^3-b\mu^2-2a \mu-c)
\]
under the isomorphism
\[
\varphi(e_1)=1,\qquad \varphi(e_2)=\mu,\qquad \varphi(e_3)=\mu^2-b\mu-a.
\]

The associativity condition
\[
(e_2^2)e_3 = e_2(e_2e_3)
\]
takes the form (let $x:=t^2$ and $y:=t^3$)
\begin{equation}\label{assoc_potential}
a^2-d-bc=f^2_{xxy}-f_{yyy}-f_{xxx}f_{xyy}=0.
\end{equation}

In the case when
\[
f(t^2,t^3) =- \frac{(t^2)^4}{96}\phi(t^3),
\]
the associativity condition \eqref{assoc_potential} becomes the Chazy III equation:
\[
\phi'''=\phi\phi''-\frac{3}{2}(\phi')^2.
\]

Moreover, every (not necessarily associative) algebra with basis $e_1$, $e_2$, and $e_3$, with unit $e_1$, given by the multiplication table \eqref{mult_table_2}, is associative if and only if the following condition holds:
\[
a^2-d-bc=0.
\]

Equivalently, consider the polynomial ring
\[
K=\Z[a,b,c,d]
\]
and the commutative nonassociative algebra
\[
K\langle e_1, e_2, e_3\rangle/M,
\]
where $K\langle e_1, e_2, e_3\rangle$ is the free nonassociative algebra on generators $e_1$, $e_2$, and $e_3$, and $M$ is the ideal generated by the multiplication table \eqref{mult_table_2}. Then the following statement holds.

\begin{prop}\label{minimal_ideal}
Let $I$ be the minimal ideal in $K$ with respect to inclusion such that the algebra $L\langle e_1, e_2, e_3\rangle/M$ is associative, where $L=K/I$. Then
\[
I=(a^2-d-bc).
\]
\end{prop}

From Proposition \ref{minimal_ideal} and the results of Section \ref{Curves_in_moduli_space}, we obtain the following statement.

\begin{theorem}\label{Dubrovin-Frobenius_and_2-valued}
The associativity condition
\[
4k_8=k_4^2-k_6k_2
\]
in the universal symmetric $2$-algebraic two-valued group is equivalent to the associativity of the Dubrovin--Frobenius structure on a three-dimensional domain with potential
\begin{equation}\label{the_potential}
F(t) = \frac{1}{2}(t^1)^2t^3+\frac{1}{2}t^1(t^2)^2 - \frac{(t^2)^4}{96}\phi(t^3).
\end{equation}
\end{theorem}

Let us write explicitly how the coalgebra structure looks for the Dubrovin--Frobenius structure with potential \eqref{the_potential}. Recall that if $A$ is a Frobenius algebra with dual bases $\{x_i\}$ and $\{y_i\}$ (i.e.
\[
\langle x_i, y_j\rangle = \delta_{ij}),
\]
then the counit of the corresponding coalgebra coincides with the Frobenius functional $\epsilon$, and the diagonal (comultiplication) is given by the formula (see Proposition \ref{coalgebra_strucure_in_Frobenius_algebra})
\[
\Delta(x)= \sum\limits_{i=1}^nx_ix\otimes y_i=\sum\limits_{i=1}^nx_i\otimes xy_i.
\]
From the form of the metric \eqref{metric_matrix}, it follows that the basis $(e_1,e_2,e_3)$ is dual to the basis $(e_3,e_2,e_1)$. Let
\[
a:=f_{xxy},\qquad b:=f_{xxx},\qquad c:=f_{xyy},\qquad d:=f_{yyy}.
\]
Then
\begin{equation}\label{diagonal_map}
\begin{aligned}
\Delta(e_1)&=Q=e_1\otimes e_3+e_2\otimes e_2+e_3\otimes e_1,\\
\Delta(e_2)
&=
c\,e_1\otimes e_1
+a\,(e_1\otimes e_2+e_2\otimes e_1)
+b\,e_2\otimes e_2
+e_2\otimes e_3
+e_3\otimes e_2,\\
\Delta(e_3)
&=
d\,e_1\otimes e_1
+c\,(e_1\otimes e_2+e_2\otimes e_1)
+a\,e_2\otimes e_2
+e_3\otimes e_3,
\end{aligned}
\end{equation}
where $e_k=\frac{\partial}{\partial t^k}$, $e=e_1$, and $Q$ is the Casimir element, which gives a solution of the quantum Yang--Baxter equation corresponding to the Frobenius algebra (see Section \ref{Quantum_Yang_Baxter_and_Dubrovin_Frobenius_structures}).

\subsection{QYBE and Dubrovin--Frobenius Structures}\label{Quantum_Yang_Baxter_and_Dubrovin_Frobenius_structures}

In \cite[Theorem 3.4]{Beidar97}, it was shown that for every Frobenius algebra $A$ (not necessarily commutative) over a commutative ring $K$ with dual bases $\{e_i\}_{i=1}^n$ and $\{e^i\}_{i=1}^n$, there is a canonical element
\[
Q=\sum\limits_{i=1}^n e_i\otimes e^i,
\]
called the Casimir element, such that $Q$ satisfies the braid relation
\[
Q_{12}Q_{23}Q_{12}=Q_{23}Q_{12}Q_{23},
\]
where the following standard notation is used:
\[
Q_{12}=\sum_{i=1}^n e_i\otimes e^i\otimes 1,\quad
Q_{23}=\sum_{i=1}^n 1\otimes e_i\otimes e^i,\quad
Q_{13}=\sum_{i=1}^n e_i\otimes 1\otimes e^i.
\]

Let $T: A\otimes A\to A\otimes A$ denote the transposition operator (also called the ``flip'')
\[
T(a\otimes b)=b\otimes a.
\]
Further, let
\[
R=QT\in\End(A\otimes A)
\]
be the ``$R$-matrix''. Then $R$ satisfies the quantum Yang--Baxter equation:
\begin{equation}\label{quantum_Yang-Baxter_equation}
R_{12}R_{13}R_{23}=R_{23}R_{13}R_{12}.
\end{equation}

Let us return to the Dubrovin--Frobenius structure from Theorem \ref{Dubrovin-Frobenius_and_2-valued}. We know that the basis $(e_1,e_2,e_3)$ is dual to the basis $(e_3,e_2,e_1)$ with respect to the metric \eqref{metric_matrix}. Let us write the Casimir element:
\begin{equation}\label{Casimir}
Q=\sum_{i=1}^3 e_i\otimes e^i \;=\; e_1\otimes e_3+e_2\otimes e_2+e_3\otimes e_1.
\end{equation}

Introduce a matrix:
\begin{equation}\label{R_matrix_2}
R =
\begin{pmatrix}
0 & c & d & c & a^{2} & ac & d & ac & c^{2} \\
0 & a & c & a & ab+c & bc+d & c & a^{2} & ac \\
1 & 0 & 0 & 0 & a & c & 0 & c & d \\
0 & a & c & a & ab+c & a^{2} & c & bc+d & ac \\
1 & b & a & b & 2a+b^{2} & ab+c & a & ab+c & a^{2} \\
0 & 1 & 0 & 1 & b & a & 0 & a & c \\
1 & 0 & 0 & 0 & a & c & 0 & c & d \\
0 & 1 & 0 & 1 & b & a & 0 & a & c \\
0 & 0 & 1 & 0 & 1 & 0 & 1 & 0 & 0
\end{pmatrix},
\end{equation}

In the basis
\[
\bigl(e_1\!\otimes\!e_1,\ e_1\!\otimes\!e_2,\ e_1\!\otimes\!e_3,\ e_2\!\otimes\!e_1,\ e_2\!\otimes\!e_2,\ e_2\!\otimes\!e_3,\ e_3\!\otimes\!e_1,\ e_3\!\otimes\!e_2,\ e_3\!\otimes\!e_3\bigr)
\]
the $R$-matrix takes the form \eqref{R_matrix_2}, where
\[
a:=f_{xxy},\qquad b:=f_{xxx},\qquad c:=f_{xyy},\qquad d:=f_{yyy}.
\]

In fact, the following more general statement holds.

\begin{theorem}\label{general_Yang_Baxter_and_associativity_condition}
The matrix \eqref{R_matrix_2} satisfies the Yang--Baxter equation
\[
R_{12}R_{13}R_{23}=R_{23}R_{13}R_{12}
\]
if and only if the following condition holds:
\[
a^2-d-bc=0
\]
\end{theorem}

\begin{proof}
Let
\[
Y:= R_{12}R_{13}R_{23}-R_{23}R_{13}R_{12}.
\]
Then a direct computation shows that
\[
Y=(a^2-d-bc)N,
\]
where $N = N(a,b,c,d)$ is a $(27\times 27)$-matrix over the ring $\Z[a,b,c,d]$ with some entries equal to 1 (hence, $N\not\equiv 0$).
\end{proof}

From the results of Section \ref{Curves_in_moduli_space}, it follows that there are solutions of the Yang--Baxter equation of the form \eqref{R_matrix_2} whose entries admit parametrization by quasimodular forms.

\begin{coroll_theorem}\label{Yang_Baxter_and_associativity_condition}
The associativity condition
\[
4k_8=k_4^2-k_6k_2
\]
in the universal symmetric $2$-algebraic two-valued group is equivalent to the quantum Yang--Baxter equation for the Dubrovin--Frobenius structure on a three-dimensional domain with potential
\[
F(t) = \frac{1}{2}(t^1)^2t^3+\frac{1}{2}t^1(t^2)^2 - \frac{(t^2)^4}{16}\phi(t^3)
\]
and Casimir element \eqref{Casimir}.
\end{coroll_theorem}

Note that Proposition \ref{coalgebra_strucure_in_Frobenius_algebra} implies that for any Frobenius algebra $(A,\theta)$, for the corresponding coalgebra $(A,\Delta,\theta)$ the value of the diagonal $\Delta$ on the unit equals the Casimir element $Q$. In particular, the coalgebra $(A,\Delta,\theta)$ with $\dim A>1$ cannot be a Hopf algebra, since in that case one would have
\[
\Delta(e)=e\otimes e.
\]

\section{Conclusion and Open Problems}\label{conclusion}

The main result of this paper is the proof of equivalence of five conditions unified in
diagram~\eqref{diagram_of_equivalences}: the associativity condition for the universal symmetric
$2$-algebraic 2-valued group, the Chazy~III equation, the horizontality of the Ramanujan
vector field, the associativity of the 3D Dubrovin--Frobenius structure with potential \eqref{the_potential},
and the quantum Yang--Baxter equation for the corresponding Casimir element. All five
equivalences reduce to the single algebraic relation
\[
  4k_8 = k_4^2 - k_6 k_2,
\]
which is realized geometrically through the orbits of the $\mathrm{SL}_2(\mathbb{C})$-action
on the solution space of the Chazy equation.

We emphasize that  diagram~\eqref{diagram_of_equivalences} is established specifically in the 2-valued
case. Below we formulate concrete problems arising naturally from the results obtained, and
indicate in each case what is known and what is missing for their resolution.

%\subsection*{An analogue of the Chazy equation for $n=3$}
%In~\cite{BK25b}, symmetric $3$-algebraic three-valued groups $G_{\mathbb{C}}(P_3)$ are
%classified. The associativity condition for $n=3$ is a system of polynomial equations on the
%coefficients of the polynomial $P_3(z;x,y)$, analogous to relation~\eqref{eq:assoc}. The
%concrete question is: \emph{does this condition, under a suitable parametrisation of the
%coefficients by curves in the moduli space, become equivalent to some autonomous
%third-order ODE system?} Natural candidates are the Chazy~XII equation or the
%Darboux--Halphen system, both of which carry an $\mathrm{SL}_3(\mathbb{C})$-symmetry;
%however, a direct computation in the style of the proof of Theorem~3 for $n=3$ has not
%yet been carried out.

\subsection*{$n$-Valued Analogues for $n>2$}
The first problem is to understand whether the picture developed here admits a genuine analogue for $n$-valued groups with $n>2$. In particular, one may ask whether the associativity equations in the $n=3$ case lead to a distinguished differential system playing a role analogous to that of the Chazy III equation for $n=2$. At present this should be regarded as an open problem rather than a prediction, although the representation-theoretic features of the $n=2$ case suggest that an $\SL_3(\mathbb C)$-equivariant differential system may arise.

\subsection*{Modular and Arithmetic Aspects}
In the 2-valued case, the parameters $k_2(\tau)$, $k_4(\tau)$ and $k_6(\tau)$ are expressed, for the appropriate normalization, through the quasimodular forms $E_2(\tau)$, $E_2'(\tau)$ and $E_2''(\tau)$. This suggests that the parameter space of the theory has a natural modular interpretation. It would be interesting to determine to what extent analogous moduli spaces for higher-valued groups carry modular or locally symmetric structures, and how the degenerate solutions should be interpreted from this point of view.

%\subsection*{Arithmetic properties of the $R$-matrix}
%The matrix $R = R(a,b,c,d)$ from Theorem~8, upon specialization
%\[
%  a = a(\tau),\quad b = b(\tau),\quad c = c(\tau),\quad d = d(\tau)
%\]
%by quasimodular forms, satisfies the QYBE. The concrete problem is: \emph{describe
%which values $\tau\in\mathbb{H}$ yield $R$-matrices with algebraic $($or rational$)$
%entries, and relate this to the theory of complex multiplication for the elliptic
%curves~\eqref{eq:elliptic}.} This is concretely verifiable for CM-points
%$\tau = i,\, e^{2\pi i/3},\ldots$ by explicit computation using the formulae already
%available in the paper.

\subsection*{Quantization of the Structure}
The $R$-matrix~\eqref{R_matrix_2} is a solution of the QYBE over a ring $A=\Z[a,b,c,d]$. The natural question is: does there exist a one-parameter deformation $R_q$ of $R(a,b,c,d)$
satisfying the quantum Yang--Baxter equation
\[
  R_{12}R_{13}R_{23} = R_{23}R_{13}R_{12}
\]
over some non-commutative ring $A_q$ such that such that the specialization at $q=1$ is $A$.

%for $q\neq 1$ and reducing to $R$ as $q\to 1$, while remaining compatible with the
%differential constraints imposed by the Chazy equation?} The principal obstruction is
%that the Chazy equation is a differential, not a $q$-differential, equation, and its
%$q$-analogue is not known.

\subsection*{Spectral and Integrable Aspects}
The appearance of the Chazy equation in several areas of integrable systems suggests that the present results may admit an interpretation in terms of spectral problems. A particularly interesting question is whether the Chazy side of the equivalence diagram can be related more directly to the ODE/IM correspondence or to functional relations for spectral determinants. Establishing such a link would require substantial new work, but it could provide a different perspective on the associativity equations studied here.

\subsection*{Discriminants and Singularity Theory}
The vector fields $\mathsf{u}$, $\mathsf{v}$ and $\mathsf{w}$ on the base of the universal elliptic family are tangent to the discriminant locus $\Delta=0$. This suggests a link between the present construction and the geometry of discriminants. It would therefore be worthwhile to investigate whether the 2-valued group viewpoint can be incorporated into the framework of singularity theory, for example in connection with flat structures on the complement of the discriminant or with Saito-type constructions.

\subsection*{Operadic Formulation}
The notion of an $n$-valued algebraic structure suggests an operadic or PROP-type reformulation. Such a language could help separate the genuinely algebraic content of multi-valued associativity from the special analytic realizations discussed in this paper. It may also clarify whether the associativity constraints admit homotopy-theoretic generalizations.

%To summarize, the results of this paper do not yet provide a complete theory of multivalued algebraic structures. They do, however, show that in the two-valued case the associativity equations, the Chazy III equation, quasimodular geometry, Gauss--Manin connections, and Frobenius-type structures are linked by a common and remarkably rigid differential pattern. We hope that this point of view will serve as a useful starting point for further work.

\subsection*{Higher-dimensional $n$-Valued Groups and Abelian Varieties}

Finally, we consider it important to develop the theory of $n$-valued algebraic groups associated with automorphisms of Jacobians of algebraic curves. Results in this direction will connect the theory of $n$-valued groups with the theory of multidimensional Abelian functions, the method of finite-gap integration, and the theory of discriminants.

%\clearpage
%\addcontentsline{toc}{section}{References}
\bibliographystyle{bib_style}
{\small \bibliography{data}}

\hfill\\

\noindent\textit{Victor Buchstaber}\\
Steklov Mathematical Institute of Russian Academy of Sciences\\
Email: \texttt{buchstab@mi-ras.ru}

\hfill\\

\noindent\textit{Mikhail Kornev}\\
Steklov Mathematical Institute of Russian Academy of Sciences\\
Email: \texttt{mkorneff@mi-ras.ru}

\hfill\\

\noindent\textit{Vladimir Rubtsov}\\
LAREMA, UMR 6093 du CNRS\\
University of Angers\\
Email: \texttt{volodya@univ-angers.fr}

\end{document}